\documentclass[11pt,fleqn]{article}
\usepackage{amsfonts}
\usepackage{amssymb}
\usepackage{amsthm}
\usepackage{relsize}
\usepackage[leqno]{amsmath}
\usepackage{multicol}
\usepackage{blkarray}
\usepackage[export]{adjustbox}
\usepackage{array}
\usepackage{framed}

\makeatletter
\newcommand{\leqnomode}{\tagsleft@true}
\newcommand{\reqnomode}{\tagsleft@false}
\makeatother

\newcommand{\R}{\mathbb{R}}

\usepackage[UKenglish]{babel}
\usepackage{enumerate}
\usepackage{hyperref}
\usepackage{latexsym}
\usepackage{lmodern}

\theoremstyle{plain}
\newtheorem{theorem}{Theorem}[section]

\newtheorem{lemma}[theorem]{Lemma}
\newtheorem{corollary}[theorem]{Corollary}
\newtheorem{proposition}[theorem]{Proposition}

\newtheorem{observation}[theorem]{Observation}
\theoremstyle{remark}
\newtheorem{remark}{Remark}

\usepackage{tikz}
\usetikzlibrary{arrows,snakes,backgrounds}
\usetikzlibrary{decorations.pathmorphing}
\usepackage{graphicx}

\newtheorem{innercustomgeneric}{\customgenericname}
\providecommand{\customgenericname}{}
\newcommand{\newcustomtheorem}[2]{%
  \newenvironment{#1}[1]
  {%
   \renewcommand\customgenericname{#2}%
   \renewcommand\theinnercustomgeneric{##1}%
   \innercustomgeneric
  }
  {\endinnercustomgeneric}
}

\newcustomtheorem{customlemma}{Lemma}

\theoremstyle{plain}
\newtheorem{innercustomthm}{Theorem}
\newenvironment{customtheorem}[1]
  {\renewcommand\theinnercustomthm{#1}\innercustomthm}
  {\endinnercustomthm}

\usepackage[margin=2.25cm]{geometry}
\usepackage{marvosym}

\newcommand*\samethanks[1][\value{footnote}]{\footnotemark[#1]}

\title{A Sum of Squares Characterization of Perfect Graphs}

\author{Amir Ali Ahmadi\thanks{A. A. Ahmadi and C. Dibek are with the department of Operations Research and Financial Engineering,
		Princeton University, USA.
		Emails: {\tt\small aaa@princeton.edu}; {\tt\small cdibek@princeton.edu} \newline This work was partially supported by an AFOSR MURI award, the DARPA Young Faculty Award, the Princeton SEAS Innovation Award, the NSF CAREER Award, the Google Faculty Award, and the Sloan Fellowship.} \ \ \ \ \ \ \ \ \
Cemil Dibek\samethanks
}

\date{}
\begin{document}
\maketitle

\begin{abstract}
We present an algebraic characterization of perfect graphs, i.e., graphs for which the clique number and the chromatic number coincide for every induced subgraph. We show that a graph is perfect if and only if certain nonnegative polynomials associated with the graph are sums of squares. As a byproduct, we obtain several infinite families of nonnegative polynomials that are not sums of squares through graph-theoretic constructions. We also characterize graphs for which the associated polynomials belong to certain structured subsets of sum of squares polynomials. Finally, we reformulate some well-known results from the theory of perfect graphs as statements about sum of squares proofs of nonnegativity of certain polynomials.\\

\end{abstract}

\reqnomode
\section{Introduction} \label{sec:intro}

A graph is \emph{perfect} if for each of its induced subgraphs, the chromatic number equals the cardinality of a largest clique. Perfect graphs, introduced by Berge in 1960, have elegant theoretical properties and curious connections with linear, integer, and semidefinite programming. For instance, perfect graphs appear in the study of exactness of linear programming relaxations of integer programs. As an example, for a matrix $A \in \{0,1\}^{m \times n}$, all vertices of the polytope $\{x \in \R^n : Ax \leq 1, x \geq 0\}$ are integral if and only if the undominated rows of $A$ are the incidence vectors of the maximal cliques of a perfect graph~\cite{CRST, Chvatal}. Moreover, several combinatorial problems that are NP-hard on general graphs can be solved efficiently on perfect graphs using semidefinite programming \cite{GLS}. Examples include the maximum independent set and the minimum clique cover problems. More generally, perfect graphs have been the subject of much research in recent decades due to the fact that they are at the crossroad of several mathematical disciplines, including graph theory, information theory, combinatorial optimization, polyhedral and convex geometry, and semidefinite programming~\cite{Chvatal, CKLMS, Fulkerson, GLS, GLS2, Lovasz, GoPaTh}.

The second notion of interest to this paper is that of sum of squares polynomials. A polynomial is a \emph{sum of squares} (sos) if it can be written as a sum of squares of some other polynomials. There has been a growing interest in sos polynomials recently due to the fact that they provide semidefinite programming-based sufficient conditions for problems involving nonnegative polynomials. It is well known that several important problems in applied and computational mathematics can be formulated as optimization problems over the set of nonnegative polynomials. Although these problems are generally intractable to solve exactly, they can be efficiently approximated by replacing nonnegativity constraints with sum of squares requirements. By connecting ideas from real algebraic geometry and semidefinite programming, sum of squares polynomials have significantly impacted both discrete and continuous optimization over the last two decades; see, e.g.,~\cite{Lasserre2, ParriloMP, BPT, LasserreBook, ghsurvey, Laurentsurvey}.

In this work, we introduce and study the notion of sos-perfectness, a notion that brings together perfect graphs and sos polynomials. For a graph $G = (V, E)$ with clique number $\omega(G)$, we define the following quartic (homogeneous) polynomial in the variables $x = (x_1, \dots, x_{|V(G)|})^T$:
\begin{equation}
\begin{aligned}
\hspace{2.9cm} 
p_G(x) \mathrel{\mathop:}= -2 \: \omega(G) \mathlarger{\sum}\limits_{ij \in E(G)} x_i^2x_j^2 + (\omega(G) - 1) \left( \mathlarger{\sum}\limits_{i=1}^{|V(G)|} x_i^2 \right)^2.
\end{aligned}
\label{eqn:our_polynomial}
\end{equation}
It turns out that for every graph $G$, the polynomial $p_G(x)$ is nonnegative by construction. We say that a graph $G$ is \emph{sos-perfect} if $p_H(x)$ is sos for every induced subgraph $H$ of $G$. In Section~\ref{sec:sos_perfect} of this paper, we prove the following theorem.

\begin{theorem}
A graph is perfect if and only if it is sos-perfect.
\label{thm:main_thm}
\end{theorem}

The remainder of this paper is organized as follows. In Section~\ref{sec:preliminaries}, we recall some definitions and results related to perfect graphs and sos polynomials. In Section~\ref{sec:sos_perfect}, we prove Theorem \ref{thm:main_thm} (without using the strong perfect graph theorem; see Section~\ref{sec:preliminaries} for the statement of this theorem and also Remark~\ref{rem:no_spgt}). Our proof brings together a number of interesting existing results in graph theory and conic optimization. In Section~\ref{sec:nonnegative_but_not_sos}, we focus on the connection between imperfect graphs and nonnegative polynomials that are not sos. In Sections~\ref{subsec:graphs} and~\ref{subsec:graph_operations}, we provide several infinite families of nonnegative polynomials that are not sos through various graph-theoretic constructions. In Section~\ref{subsec:random_graphs}, by building on previous results on the probable values of certain parameters associated with Erd\H{o}s-R\'enyi random graphs $G_{n,p}$, we show that for a fixed parameter $p$ and for large enough $n$, the polynomial $p_{G_{n,p}}(x)$ is nonnegative but not sos with high probability. In Section~\ref{subsec:separating_hyper}, we provide an explicit hyperplane that separates a given non-sos polynomial $p_G(x)$ from the set of sos polynomials. In Section~\ref{subsec:convexity}, we show that an example of a convex nonnegative polynomial that is not sos cannot arise from our graph-theoretic constructions. The construction of such a polynomial was an open problem until recently \cite{Saunderson}. In Section~\ref{sec:smaller_subsets_of_Cn}, we examine certain subsets of sos polynomials which admit a linear or second-order cone representation~\cite{AhmMaj}, or a more restricted semidefinite representation. We study the bounds that optimization over these subsets produces on the clique number of a graph, and characterize the graphs for which these bounds are tight for all induced subgraphs. Finally, in Section~\ref{sec:revisit_perfect}, we reformulate a number of results from the theory of perfect graphs as statements about sum of squares proofs of nonnegativity of certain polynomials. Our hope is that our reformulations will lead to more connections between structural graph theory and real algebraic geometry, and ideally to an algebraic proof of the strong perfect graph theorem. As a step in this direction, we use one of our corollaries together with results from linear algebra to give a short proof of the weaker statement that graphs with no odd cycles of length 5 or more are perfect.

\section{Preliminaries} \label{sec:preliminaries}

A (multivariate) \emph{polynomial} $p(x)$ in variables $x \mathrel{\mathop:}= (x_1, \dots, x_n)^T$ is a function from $\mathbb{R}^n$ to $\mathbb{R}$ that is a finite linear combination of monomials:
$$p(x) = \sum\limits_{\alpha} c_{\alpha}x^{\alpha} = \sum\limits_{\alpha_1, \dots, \alpha_n} c_{\alpha_1, \dots, \alpha_n} x_1^{\alpha_1} \dots x_n^{\alpha_n},$$
where the sum is over $n$-tuples of nonnegative integers $\alpha_i$. The \emph{degree} of a monomial $x^{\alpha}$ is equal to $\alpha_1 + \dots + \alpha_n$. The degree of a polynomial $p(x)$ is defined to be the highest degree of its monomials. A \emph{form} (or a homogeneous polynomial) is a polynomial where all the monomials have the same degree.

We denote the set of real symmetric $n \times n$ matrices by $S_n$. A matrix $M \in S_n$ is \emph{positive semidefinite} (psd) if $x^T Mx \geq 0$ for all $x \in \R^n$. A matrix $M \in S_n$ is \emph{nonnegative} if the entries of $M$ are all nonnegative. We write $M \succeq 0$ if $M$ is psd, and $M \geq 0$ if $M$ is nonnegative. We denote the set of $n \times n$ psd (resp. nonnegative) matrices by $S_n^+$ (resp. $N_n$). The trace of $M$ is denoted by $\text{Tr}(M)$.

All graphs in this paper are undirected, finite, and simple (i.e., have no loops or parallel edges). Throughout the paper, $G = (V,E)$ denotes a graph with vertex set $V(G)$ and edge set $E(G)$. The complement of a graph $G$, denoted by $\overline{G}$, is the graph with vertex set $V(G)$ and edge set consisting of all distinct pairs of vertices that are not adjacent in $G$. The matrices $A \mathrel{\mathop:}= A_{G}$ and $\overline{A} \mathrel{\mathop:}= A_{\overline{G}}$ respectively denote the adjacency matrices of $G$ and $\overline{G}$. We drop the subscript when the graph in consideration is clear from the context.
The matrices $I$ and $J$ respectively denote the identity matrix and the all-ones matrix. Observe that $A + \overline{A} + I = J$.

A graph $H$ is an \emph{induced subgraph} of a graph $G$ if $V(H) \subseteq V(G)$ and any two vertices of $H$ are adjacent if and only if they are adjacent in $G$. We say that $G$ \emph{contains} a graph $H$ if $G$ has an induced subgraph isomorphic to $H$, and that $G$ is \emph{$H$-free} if it does not contain $H$. For an integer $k \geq 4$, a {\em hole} (of length $k$) is a graph isomorphic to the chordless $k$-vertex cycle $C_k$, and an {\em antihole} (of length $k$) is a graph isomorphic to $\overline{C}_k$. A hole (or an antihole) is {\em odd} if its length is odd.

A \emph{clique} in a graph is a set of pairwise adjacent vertices, and an \emph{independent set} is a set of pairwise non-adjacent vertices. The clique number of a graph $G$, denoted by $\omega(G)$, is the size of a maximum clique in $G$, and the independence number of $G$, denoted by $\alpha(G)$, is the size of a maximum independent set in $G$. The chromatic number of a graph $G$, denoted by $\chi(G)$, is the smallest integer $\ell \geq 1$ such that $V(G)$ can be partitioned into $\ell$ independent sets.
Every graph $G$ clearly satisfies $\chi(G) \geq \omega(G)$. The inequality, however, may be strict. For instance, if $G$ is an odd hole, it is easy to see that $\omega(G) = 2$ and $\chi(G) = 3$. Similarly, if $G$ is an odd antihole of length $2k+1$ for some $k \geq 2$, then $\omega(G) = k$ and $\chi(G) = k+1$.

\subsection{Perfect graphs} \label{subsec:perfect_graphs}
A graph $G$ is \emph{perfect} if every induced subgraph $H$ of $G$ satisfies $\chi(H) = \omega(H)$. Berge introduced perfect graphs and made two conjectures~\cite{Berge61}. The first, proved by Lov\'asz~\cite{Lovasz2} and now known as the weak perfect graph theorem, states that a graph is perfect if and only if its complement is perfect. Berge's second conjecture characterizes minimal imperfect graphs. A graph $G$ is \emph{minimal imperfect} if $G$ is not perfect but every proper induced subgraph of $G$ is perfect. Berge observed that odd holes and odd antiholes are minimal imperfect graphs, and conjectured that they are, in fact, \emph{the only} minimal imperfect graphs. This conjecture, now known as the strong perfect graph theorem, was proved by Chudnovsky, Robertson, Seymour, and Thomas~\cite{SPGT}: A graph is perfect if and only if it does not contain an odd hole or an odd antihole. Indeed, since odd holes and odd antiholes satisfy $\chi(G) = \omega(G) + 1$, perfect graphs do not contain odd holes or odd antiholes. The proof of the converse direction is long and relies on structural graph theory.

The \emph{theta number} of a graph $G$, introduced by Lov\'asz \cite{Lovasz} and denoted by $\vartheta(G)$, is given as the optimal value of the following semidefinite program:
\begin{equation}
\begin{aligned}
\hspace{4cm} 
\vartheta(G) \:\: \mathrel{\mathop:}= \:\:\:\:\:\:
& \max\limits_{X \in S_n}
& & \text{Tr}(JX) \\
& \text{s.t.}
&& X_{ij} = 0 \:\:\:\: \mathrm{if} \:\: ij \in E \\
&&&  \text{Tr}(X) = 1 \\
&&& X \succeq 0.
\end{aligned}
\label{eqn:lov_number}
\end{equation}

Gr\"{o}tschel, Lov\'asz, and Schrijver \cite{GLS} showed that the theta number of the complement of a graph is sandwiched between the clique number and the chromatic number of the graph, that is, for any graph $G$, we have $\omega(G) \leq \vartheta(\overline{G}) \leq \chi(G)$. Since $\vartheta(\overline{G})$ can be computed with arbitrary precision in polynomial time via the semidefinite program \eqref{eqn:lov_number}, one of the consequences of this result is that for a perfect graph $G$, the clique number of $G$ can be computed in polynomial time.

A strengthening of the theta number was introduced by McEliece et al.~\cite{MRR} and Schrijver~\cite{Schrijver}, where an entry-wise nonnegativity constraint on the matrix $X$ is added to \eqref{eqn:lov_number}:
\begin{equation}
\begin{aligned}
\hspace{4cm} \vartheta'(G) \:\: \mathrel{\mathop:}= \:\:\:\:\:\:
& \max\limits_{X \in S_n}
& & \text{Tr}(JX) \\
& \text{s.t.}
&& X_{ij} = 0 \:\:\:\: \mathrm{if} \:\: ij \in E \\
&&&  \text{Tr}(X) = 1 \\
&&& X \succeq 0 \\
&&& X \geq 0.
\end{aligned}
\label{eqn:sch_number}
\end{equation}
Schrijver \cite{Schrijver} observed that $\vartheta'(\overline{G})$, too, is an upper bound on $\omega(G)$, that is, for any graph $G$, we have 
\begin{equation}
\begin{aligned}
\hspace{5cm}
\omega(G) \leq \vartheta'(\overline{G}) \leq \vartheta(\overline{G}) \leq \chi(G).
\end{aligned}
\label{eqn:sandwich}
\end{equation}

\subsection{Sum of squares polynomials}\label{subsec:sos}

A polynomial $p: \R^n \rightarrow \R$ with real coefficients is \emph{nonnegative} if $p(x) \geq 0$ for all $x \in \R^n$ and a \emph{sum of squares} (sos) if there exist polynomials $q_1(x),\ldots,q_m(x)$ such that $p(x) = \sum_{i=1}^m q_i^2(x)$. While every sos polynomial is clearly nonnegative, Hilbert showed in 1888 that there exist nonnegative polynomials that are not sos \cite{Hilbert}. His proof was not constructive and did not lead to an explicit example of a nonnegative polynomial that is not sos. The first examples of such polynomials were found by Motzkin~\cite{Motzkin} and Robinson~\cite{Robinson} nearly eighty years after Hilbert's proof.

From a complexity standpoint, testing nonnegativity of polynomials of any fixed degree $2d \geq 4$ is NP-hard~\cite{MurtKaba}. By contrast, checking whether a polynomial is sos can be done by solving a semidefinite program. Indeed, a polynomial $p(x)$ in $n$ variables and of degree $2d$ is sos if and only if there exists a psd matrix $Q$ such that $p(x) = z(x)^TQz(x)$, where $z(x)$ is the vector of monomials of degree up to $d$, i.e., $z(x) = (1, x_1, \dots, x_n, x_1x_2, \dots, x_n^d)^T$ (see, e.g.,~\cite{ChoiLamRez,ParriloThesis}). In fact, this statement leads to a semidefinite programming-based approach for optimizing a linear function over the intersection of the set of sos polynomials with an affine subspace. This observation has enabled wide-ranging applications, see, e.g., \cite{ghsurvey}.

\section{A Sum of Squares Characterization of Perfect Graphs}\label{sec:sos_perfect}

In this section, we prove Theorem~\ref{thm:main_thm} by establishing a few intermediary lemmas. We begin by stating some relevant results from prior literature.

A matrix $M \in S_n$ is \emph{copositive} if $x^TMx \geq 0$ for all $x \geq 0$ (i.e., for all vectors $x$ in the nonnegative orthant). We denote the set of $n \times n$ copositive matrices by $\mathcal{C}_n$. It is not difficult to observe that minimization of a quadratic function over the standard simplex $\Delta \mathrel{\mathop:}= \{x \in \mathbb{R}^n : \sum_{i=1}^n x_i = 1, x \geq 0\}$ is equivalent to optimization of a linear function over $\mathcal{C}_n$ (see, e.g.,~\cite{DKLP, BDKRQT}):
\begin{equation}
\hspace{4.5cm}
\begin{aligned}
\min\limits_{x \in \Delta} \:\: x^TQx \:\:\:\: = \:\:\:\:
& \max\limits_{k \in \R}
& & k \\
& \text{s.t.}
&& Q - kJ \in \mathcal{C}_n.
\end{aligned}
\label{eqn:quad_copos}
\end{equation}
As shown by Motzkin and Straus \cite{MS}, the problem on the left can be related to the clique number of a graph as follows:
\begin{equation}
\hspace{5cm}
\begin{aligned}
\frac{1}{\omega(G)} \:\: = \:\:\:\:
\min\limits_{x \in \Delta} \:\:\: x^T (I + \overline{A}) x.
\end{aligned}
\label{eqn:MS}
\end{equation}
Here, $\overline{A}$ denotes the adjacency matrix of the complement graph $\overline{G}$. It follows from \eqref{eqn:quad_copos} and \eqref{eqn:MS} that
\begin{equation}
\hspace{4.6cm}
\begin{aligned}
\omega(G) \:\: = \:\:\:\:\:
& \min\limits_{k \in \R}
& & k \\
& \text{s.t.}
&& k(I + \overline{A}) - J \in \mathcal{C}_n.
\end{aligned}
\label{eqn:omega_copos}
\end{equation}
It is easy to see that a matrix $M$ belongs to $\mathcal{C}_n$ if and only if the quartic form 
\begin{equation}
\hspace{5.6cm}
\begin{aligned}
p_M(x) \mathrel{\mathop:}= \sum\limits_{i,j=1}^n M_{ij}x_i^2x_j^2
\end{aligned}
\label{eqn:p_M_x}
\end{equation}
is nonnegative. Let $\mathcal{K}_n$ denote the set of matrices $M \in S_n$ such that $p_M(x)$ is sos. Clearly, $\mathcal{K}_n \subseteq \mathcal{C}_n$. Hence, a tractable upper bound on $\omega(G)$ can be obtained by replacing $\mathcal{C}_n$ in~\eqref{eqn:omega_copos} with $\mathcal{K}_n$. Using in part a result from~\cite[Section 5]{ParriloThesis} (see also \cite[Lemma 3.5]{ChoiLam1}), 
De Klerk and Pasechnik \cite{DKLP} showed that the resulting upper bound coincides with the parameter $\vartheta'(\overline{G})$ defined in \eqref{eqn:sch_number}:
\begin{equation}
\hspace{4.6cm}
\begin{aligned}
\vartheta'(\overline{G}) \:\: = \:\:\:\:
& \min\limits_{k \in \R}
& & k \\
& \text{s.t.}
&& k(I + \overline{A}) - J \in \mathcal{K}_n.
\end{aligned}
\label{eqn:theta_sos}
\end{equation}

The following lemma sheds light on the construction of the polynomial $p_G(x)$ in \eqref{eqn:our_polynomial}, and a corollary of it will be used in the proof of Theorem~\ref{thm:main_thm}. Let us consider a more general family of quartic forms by replacing the constant $\omega(G)$ in $p_G(x)$ with an arbitrary scalar $k \in \R$:
\begin{equation}
\hspace{3.2cm}
p_{G,k}(x) \mathrel{\mathop:}= -2 k \mathlarger{\sum}\limits_{ij \in E(G)} x_i^2x_j^2 + (k - 1) \left( \mathlarger{\sum}\limits_{i=1}^{|V(G)|} x_i^2 \right)^2.
\label{eqn:pgkx}
\end{equation}

\begin{lemma}
\label{lem:nonneg_if_clique}
For any graph $G$, 
\begin{enumerate}[(a)]
\itemsep0em
\item\label{part_a} the polynomial $p_{G,k}(x)$ is nonnegative if and only if $k \geq \omega(G)$.
\item\label{part_b} the polynomial $p_{G,k}(x)$ is sos if and only if $k \geq \vartheta'(\overline{G})$.
\end{enumerate}
\end{lemma}

\begin{proof}
Let $G=(V,E)$ be a graph with $|V(G)| = n$ and with adjacency matrix $A$. We first claim that the polynomial $p_{G,k}(x)$ is nonnegative if and only if $k(I + \overline{A}) - J \in \mathcal{C}_n$. Observe that since $I + \overline{A} = J - A$, we have $k(I + \overline{A}) - J = - k A + (k -1)J$. Therefore,
\begin{equation*}
\begin{aligned}
\hspace{1cm}
k(I + \overline{A}) - J \in \mathcal{C}_n & \iff \sum\limits_{i,j=1}^n \Big [- k A + (k -1)J \Big ]_{ij}x_i^2x_j^2 \geq 0 \quad \forall x \in \R^n\\
& \iff -k \sum\limits_{i,j=1}^n A_{ij} x_i^2x_j^2 + (k -1) \sum\limits_{i,j=1}^n J_{ij} x_i^2x_j^2 \geq 0 \quad \forall x \in \R^n\\
& \iff -2k \mathlarger{\sum}\limits_{ij \in E(G)} x_i^2x_j^2 + (k - 1) \left( \mathlarger{\sum}\limits_{i=1}^{n} x_i^2 \right)^2 \geq 0 \quad \forall x \in \R^n\\
& \iff p_{G,k}(x) \geq 0 \quad \forall x \in \R^n.
\end{aligned}
\end{equation*}
Similarly, $k(I + \overline{A}) - J \in \mathcal{K}_n$ if and only if $p_{G,k}(x)$ is sos. Hence, by \eqref{eqn:omega_copos} and \eqref{eqn:theta_sos}, we obtain the following formulations of $\omega(G)$ and $\vartheta'(\overline{G})$:
\vspace{-0.6cm}
\begin{multicols}{2}
\begin{equation*}
\hspace{0.8cm}
\begin{aligned}
\omega(G) \:\: = \:\:\:\:
& \min\limits_{k \in \R}
& & k \\
& \text{s.t.}
&& p_{G,k}(x) \text{ is nonnegative,}
\end{aligned}
\end{equation*}
\hspace{0.5cm}
\begin{equation*}
\begin{aligned}
\vartheta'(\overline{G}) \:\: = \:\:\:\:
& \min\limits_{k \in \R}
& & k \\
& \text{s.t.}
&& p_{G,k}(x) \text{ is sos.}
\end{aligned}
\end{equation*}
\end{multicols}

Therefore, if $p_{G,k}(x)$ is nonnegative, then $k \geq \omega(G)$. Similarly, if $p_{G,k}(x)$ is sos, then $k \geq \vartheta'(\overline{G})$. Observe also that if $p_{G,k}(x)$ is nonnegative (resp. sos) for some $k \in \R$, then $p_{G,k'}(x)$ is nonnegative (resp. sos) for every $k' \in \R$ with $k' \geq k$. This is simply because
$$\big( k'(I + \overline{A}) - J \big) - \big( k(I + \overline{A}) - J \big) = (k' - k) (I + \overline{A})$$
is a nonnegative matrix, thus belongs to $\mathcal{C}_n$ (resp. $\mathcal{K}_n$). It follows that $p_{G,k}(x)$ is nonnegative if and only if $k \geq \omega(G)$, and that $p_{G,k}(x)$ is sos if and only if $k \geq \vartheta'(\overline{G})$.
\end{proof}

\begin{corollary}
\label{cor:pgx_nonnegative_sos}
For any graph $G$, 
\begin{enumerate}[(a)]
\itemsep0em
\item\label{part_aa} the polynomial $p_G(x)$ is nonnegative.
\item\label{part_bb} the polynomial $p_G(x)$ is sos if and only if $\omega(G) = \vartheta'(\overline{G})$.
\end{enumerate}
\end{corollary}

\begin{proof}
Part \eqref{part_aa} follows from Lemma \ref{lem:nonneg_if_clique}\eqref{part_a} since $p_G(x) = p_{G, \omega(G)} (x)$, and part \eqref{part_bb} follows from Lemma~\ref{lem:nonneg_if_clique}\eqref{part_b} and \eqref{eqn:sandwich}.
\end{proof}

We refer the reader to the recent work of Laurent and Vargas \cite{LV2}, where the equality between $\omega(G)$ and $\vartheta'(\overline{G})$ is studied for graphs with $\omega(G) \leq 2$ and for the complements of the so-called ``$\alpha$-critical'' graphs. (See also \cite{LV1, LV3} for a related study of a hierarchy of semidefinite programming-based inner approximations for $\mathcal{C}_n$, whose first level corresponds to $\mathcal{K}_n$.)

\begin{remark}
In the statement of Corollary \ref{cor:pgx_nonnegative_sos}\eqref{part_bb}, one cannot replace the quantity $\vartheta'(\overline{G})$ with the theta number $\vartheta(\overline{G})$. For example, let $G$ be the graph on 64 vertices corresponding to the vectors in $\{0,1\}^6$, with two vertices adjacent if and only if the Hamming distance between the corresponding vectors is at least 4. We have $\omega(G) = \vartheta'(\overline{G}) = 4$ and $\vartheta(\overline{G}) = 16/3$; see~\cite{Schrijver}. Therefore, $p_G(x)$ is sos by Corollary~\ref{cor:pgx_nonnegative_sos}\eqref{part_bb}, but $\omega(G) \neq \vartheta(\overline{G})$.
\end{remark}

\begin{remark}
\label{rem:when_sos}
Corollary~\ref{cor:pgx_nonnegative_sos}\eqref{part_bb} characterizes the graphs $G$ for which the polynomial $p_G(x)$ is sos. We note that the condition $\omega(G) = \vartheta'(\overline{G})$ is not enough for a graph $G$ to be perfect. In fact, even the condition $\omega(G) = \chi(G)$ is not enough for $G$ to be perfect since perfectness requires this condition to hold for every induced subgraph. As an example, let $G$ be the graph with vertex set $\{v_1, v_2, v_3, v_4, v_5, h\}$ and edge set $\{v_1v_2, v_2v_3, v_3v_4, v_4v_5, v_5v_1, hv_1, hv_2\}$, i.e., $G$ is the graph $C_5$ with an additional vertex $h$ adjacent only to $v_1$ and $v_2$. Then, it is easy to verify that $\omega(G) = \chi(G) = 3$. However, $G$ is not perfect as it contains the graph $C_5$, for which we have $2 = \omega(C_5) < \chi(C_5) = 3$. Observe also that this graph $G$ is an example of an imperfect graph for which $p_G(x)$ is sos. This observation justifies our definition of sos-perfectness which takes induced subgraphs into consideration. Indeed, if $G$ is a perfect graph, then, by definition, every induced subgraph of $G$ is perfect. However, being sos is not a ``hereditary'' property in the sense that it is possible for the polynomial $p_G(x)$ to be sos and for $G$ to have an induced subgraph $H$ with $p_H(x)$ not sos. The graph above with $H = C_5$ provides one such example.
\end{remark}

Although the condition ``$\omega(G) = \vartheta'(\overline{G})$'' is not sufficient for a graph $G$ to be perfect, we prove next that the condition ``$\omega(H) = \vartheta'(\overline{H})$ for every induced subgraph $H$ of $G$'' is. We follow a proof technique of Lov\'asz~\cite{Lovasz0} which makes use of a binary matrix associated with the maximum cliques of a graph. Let $G$ be a graph with $n$ vertices and $m$ maximum cliques. Then, the \emph{max-clique matrix} $C$ of $G$ is an $m \times n$ matrix with $C_{ij} = 1$ if the $i^{th}$ maximum clique contains the $j^{th}$ vertex, and $C_{ij} = 0$ otherwise. As an example, the max-clique matrix of the graph $C_5$ is given in Figure \ref{fig:C5}.

We prove our next lemma without using the strong perfect graph theorem (see Remark \ref{rem:no_spgt}).

\begin{figure}[h]
\hspace{2.8cm}
\begin{minipage}{0.2\textwidth}
\begin{tikzpicture}[scale=0.2]

\node[label=above:{$v_1$}, inner sep=2.5pt, fill=black, circle] at (0,4.5)(v1){}; 
\node[label=right:{$v_2$}, inner sep=2.5pt, fill=black, circle] at (3.804226, 1.2361)(v2){}; 
\node[label=below:{$v_3$}, inner sep=2.5pt, fill=black, circle] at (2.351141, -3.2361)(v3){}; 
\node[label=below:{$v_4$}, inner sep=2.5pt, fill=black, circle] at (-2.351141, -3.2361)(v4){}; 
\node[label=left:{$v_5$}, inner sep=2.5pt, fill=black, circle] at (-3.804226, 1.2361)(v5){};

\draw[black, thick] (v1) -- (v2);
\draw[black, thick] (v1) -- (v5);
\draw[black, thick] (v2) -- (v3);
\draw[black, thick] (v3) -- (v4);
\draw[black, thick] (v4) -- (v5);

\end{tikzpicture}
\end{minipage}
\begin{minipage}{0.2\textwidth}
\begin{center}
\[
\hspace{1.5cm} C=
\begin{blockarray}{cccccc}
v_1 & v_2 & v_3 & v_4 & v_5 \\
\begin{block}{[ccccc]c}
  1 & 1 & 0 & 0 & 0 & \text{clique } 1 \\
  0 & 1 & 1 & 0 & 0 & \text{clique } 2 \\
  0 & 0 & 1 & 1 & 0 & \text{clique } 3 \\
  0 & 0 & 0 & 1 & 1 & \text{clique } 4 \\
  1 & 0 & 0 & 0 & 1 & \text{clique } 5 \\
\end{block}
\end{blockarray}
 \]
\end{center}
\end{minipage}
\vspace{-0.2cm}
\caption{The graph $C_5$ and its max-clique matrix $C$}
\label{fig:C5}
\end{figure}
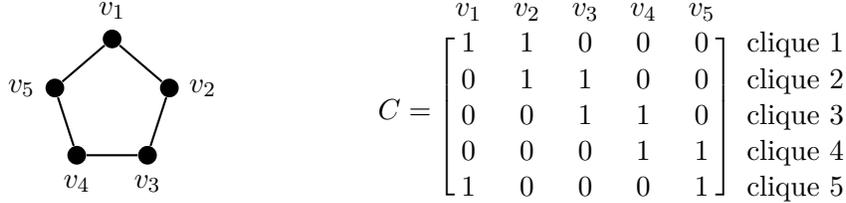

\begin{lemma}
If $G$ is a minimal imperfect graph, then $\omega(G) < \vartheta'(\overline{G})$.
\label{lem:minimal_imperfect}
\end{lemma}

\begin{proof}
We follow a proof of Lov\'asz's (see Lemma 7.9 in \cite[Section 3]{Lovasz0}). Let $G=(V,E)$ be a minimal imperfect graph with $|V(G)|=n$ and with clique number $\omega$. Let $C$ be the max-clique matrix of $G$. By a result of Padberg~\cite{Padberg} (see also~\cite{Gasparian} for a different proof), $G$ has $n$ maximum cliques, every vertex of $G$ is in exactly $\omega$ maximum cliques, and the matrix $C$ is non-singular.

Let $\lambda_1$ be the smallest eigenvalue of $C^TC$. Observe that the diagonal entries of $C^TC$ are all $\omega$. It is then not difficult to see that $\lambda_1 \in (0,\omega)$. Indeed, since $C^TC$ is psd and since $\text{Tr}(C^TC) = n \omega$, we have $\lambda_1 \in [0,\omega]$. Also, we have $\lambda_1 \neq 0$ since $C$ is non-singular, and $\lambda_1 \neq \omega$ as otherwise all eigenvalues of $C^TC$ would equal $\omega$, in which case $C^TC = \omega I$. This is a contradiction since a minimal imperfect graph has at least one edge.

Now recall that
\begin{equation}
\begin{aligned}
\hspace{4cm} \vartheta'(\overline{G}) \:\: = \:\:\:\:
& \max\limits_{X \in S_n}
& & \text{Tr}(JX) \\
& \text{s.t.}
&& X_{ij} = 0 \:\:\:\:\:\: \mathrm{if} \:\:\: ij \notin E(G) \\
&&&  \text{Tr}(X) = 1 \\
&&& X \succeq 0 \\
&&& X \geq 0.
\end{aligned}
\label{eqn:sch_number_complement}
\end{equation}
Consider the matrix $X = \frac{1}{n (\omega - \lambda_1)} (C^TC - \lambda_1 I)$. It is straightforward to check that $X$ is a feasible solution to \eqref{eqn:sch_number_complement}, and that
$$\vartheta'(\overline{G}) \geq \text{Tr}(JX) = \frac{\omega^2 - \lambda_1}{\omega - \lambda_1} > \omega. \qedhere$$
\end{proof}

\begin{corollary}
\label{cor:G_perfect}
A graph $G$ is perfect if and only if $\omega(H) = \vartheta'(\overline{H})$ for every induced subgraph $H$ of $G$.
\end{corollary}

\begin{proof}
If $G$ is perfect, then by definition, $\omega(H) = \chi(H)$ for every induced subgraph $H$ of $G$. Hence, by~\eqref{eqn:sandwich}, we have $\omega(H) = \vartheta'(\overline{H})$ for every induced subgraph $H$ of $G$. If $G$ is not perfect, then it contains a minimal imperfect graph $H_{\star}$. By Lemma \ref{lem:minimal_imperfect}, $\omega(H_{\star}) < \vartheta'(\overline{H}_{\star})$.
\end{proof}

We are now ready to present the proof of Theorem \ref{thm:main_thm}, which we restate here for ease of reference.

\begin{customtheorem}{1.1}
A graph is perfect if and only if it is sos-perfect.
\label{thm:main_thm_2}
\end{customtheorem}

\begin{proof}
By Corollary \ref{cor:G_perfect}, a graph $G$ is perfect if and only if $\omega(H) = \vartheta'(\overline{H})$ for every induced subgraph $H$ of $G$, which by Corollary~\ref{cor:pgx_nonnegative_sos} holds if and only if $p_H(x)$ is sos for every induced subgraph $H$ of $G$.
\end{proof}

\begin{remark}
\label{rem:no_spgt}
It is known that if $G$ is an odd hole or an odd antihole, then $\omega(G) < \vartheta'(\overline{G})$ (see, e.g., Proposition 15 and Proposition 19 in \cite{PVZ}, or see \cite{DKLP, LV1, LV2}).\footnote{The fact that for an odd hole or an odd antihole $G$, the inequality $\omega(G) < \vartheta'(\overline{G})$ holds also follows immediately from Lemma~\ref{lem:minimal_imperfect} since odd holes and odd antiholes are clearly minimal imperfect graphs.} Hence, \emph{assuming} the strong perfect graph theorem, one can bypass Lemma~\ref{lem:minimal_imperfect} in the proof of Theorem~\ref{thm:main_thm_2}. Indeed, if a graph $G$ is not perfect, then by the strong perfect graph theorem, it contains either an odd hole or an odd antihole, call it $H$. Since $\omega(H) < \vartheta'(\overline{H})$, by Corollary~\ref{cor:pgx_nonnegative_sos}\eqref{part_bb}, the polynomial $p_H(x)$ is not sos. Therefore, $G$ is not sos-perfect. However, we purposefully want to avoid the use of the highly-nontrivial strong perfect graph theorem in the proof of Theorem~\ref{thm:main_thm_2}. Indeed, our hope is that Theorem~\ref{thm:main_thm_2} could lead to an algebraic proof of the strong perfect graph theorem in the future (see Section~\ref{subsec:revisit_structural}).
\end{remark}

\section{Nonnegative Polynomials That Are Not Sums of Squares}\label{sec:nonnegative_but_not_sos}

As mentioned in Section~\ref{subsec:sos}, Hilbert proved the existence of nonnegative polynomials that are not sums of squares in~\cite{Hilbert}, while the first examples of such polynomials were constructed by Motzkin~\cite{Motzkin} and Robinson~\cite{Robinson} many years later. Many other examples have appeared in the literature over the years; see, e.g.,~\cite{Reznick, ChoiLam1, ChoiLam2, ChoiLamRez80}. Understanding the distinction between nonnegative polynomials and sos polynomials is an active area of research. In relatively low degrees and dimensions, constructing examples of nonnegative polynomials that are not sos seems to be a nontrivial task. 

In Section~\ref{subsec:graphs}, we provide several infinite families of nonnegative polynomials that are not sos through various families of imperfect graphs. In Section~\ref{subsec:graph_operations}, we describe certain operations on graphs that allow us to generate more nonnegative polynomials that are not sos starting from existing ones. In Section~\ref{subsec:random_graphs}, by appealing to the literature on random graph theory, we show that for a fixed parameter $p$ and for large enough $n$, the polynomial $p_{G_{n,p}}(x)$ associated with the Erd\H{o}s-R\'enyi random graph $G_{n,p}$ is nonnegative but not sos with high probability. In Section~\ref{subsec:separating_hyper}, we provide an explicit hyperplane that separates a given non-sos polynomial $p_G(x)$ from the set of sos polynomials. Finally, in Section~\ref{subsec:convexity}, we show that an example of a \emph{convex} nonnegative polynomial that is not sos cannot arise from our graph-theoretic constructions. The construction of such a polynomial was an open problem until recently~\cite{Saunderson}.

\subsection{From imperfect graphs to nonnegative polynomials that are not sos} \label{subsec:graphs}

\subsubsection{Odd holes and odd antiholes} 

Recall that odd holes and odd antiholes are minimal imperfect graphs, i.e., they are not perfect but their proper induced subgraphs are all perfect. Hence, by Lemma~\ref{lem:minimal_imperfect} and Corollary~\ref{cor:pgx_nonnegative_sos}, for every odd hole and odd antihole $G$, the polynomial $p_G(x)$ is a nonnegative polynomial that is not sos. This yields an infinite family of degree-4 polynomials that are nonnegative but not sos. As an example, consider the smallest minimal imperfect graph $C_5$ and the corresponding polynomial $p_{C_5} (x)$:
$$p_{C_5} (x) = -4 (x_1^2x_2^2+x_2^2x_3^2+x_3^2x_4^2+x_4^2x_5^2+x_1^2x_5^2) + (x_1^2+x_2^2+x_3^2+x_4^2+x_5^2)^2.$$

This polynomial is known as the ``Horn form'' in the sos community (see \cite{Reznick}). Similarly, the polynomials $p_{C_7}(x), p_{C_9}(x), \dots, p_{\overline{C}_7}(x), p_{\overline{C}_9}(x), \dots$ are all nonnegative but not sos. We recall that for $m \geq 2$, we have $\omega(C_{2m+1}) = 2$ and $\omega(\overline{C_{2m+1}}) = m$, and therefore it is immediate to explicitly write down the polynomial $p_G(x)$ when $G$ is an odd hole or an odd antihole.

Odd holes are a special case of the so-called $\alpha$-critical graphs, i.e., graphs for which the removal of any edge increases the stability number. It follows from \cite[Corollary 5.3.]{LV2} that if $G$ is the complement of an $\alpha$-critical graph and not complete multipartite (see Section \ref{subsec:dsos_sdsos} for the definition), then $\omega(G) < \vartheta'(\overline{G})$. Therefore, the fact that odd antiholes yield nonnegative polynomials that are not sos also follows from \cite[Corollary 5.3.]{LV2}.

\subsubsection{Powers of cycles and their complements}

The conclusion of Lemma \ref{lem:minimal_imperfect} holds for a more general class of graphs than minimal imperfect graphs. A graph $G$ is called \emph{partitionable} (or an \emph{$(\alpha,\omega)$-graph}) if there exist integers $\alpha \geq 2$ and $\omega \geq 2$ such that $|V(G)| = \alpha \omega +1$, and for every vertex $v$, there is a partition of $V(G) \setminus \{v\}$ into $\alpha$ many cliques of size $\omega$ and $\omega$ many independent sets of size $\alpha$. Using only elementary linear algebra, the authors in~\cite{BHT} show that every partitionable graph $G$ with $n$ vertices has $n$ maximum cliques, every vertex of $G$ is in exactly $\omega(G)$ maximum cliques, and the max-clique matrix $C$ of $G$ is non-singular. Notice that these were the only properties of minimal imperfect graphs that were used to prove Lemma~\ref{lem:minimal_imperfect}. Hence, the same proof implies that if $G$ is a partitionable graph, then $\omega(G) < \vartheta'(\overline{G})$. Consequently, by Corollary~\ref{cor:pgx_nonnegative_sos}, partitionable graphs provide an infinite family of polynomials that are nonnegative but not sos.

It can be shown that a minimal imperfect graph $G$ is partitionable with $\alpha = \alpha(G)$ and $\omega~=~\omega(G)$. There are, however, many other partitionable graphs; see, e.g., \cite{Pecher1, Pecher2} for Cayley partitionable graphs, and \cite{CGPW} for other constructions of partitionable graphs. The simplest examples are ``powers'' of cycles. The $k^{th}$ \emph{power} of a graph $G$, denoted by $G^k$, is a graph with the same vertex set as $G$, and with two vertices adjacent if and only if their distance\footnote{The distance between two vertices in a graph is the number of edges in a shortest path between them.} in $G$ is at most $k$. For every integer $\alpha \geq 2$ and $\omega \geq 2$, the graph $C_{\alpha\omega+1}^{\omega-1}$ is a partitionable graph; see~\cite{CGPW}. Notice that if $\omega = 2$, then the graphs $C_{\alpha\omega+1}^{\omega-1} = C_{2\alpha+1}$ are precisely odd holes, and if $\alpha = 2$, then the graphs $C_{\alpha\omega+1}^{\omega-1} = C_{2\omega+1}^{\omega-1} = \overline{C_{2\omega+1}}$ are precisely odd antiholes. However, for $\alpha, \omega \geq 3$, we obtain several other partitionable graphs that are not odd holes or odd antiholes, such as $C_{10}^{2}, C_{13}^{2}, C_{13}^{3}, C_{16}^{2}, C_{16}^{4}, C_{17}^{3}, \dots$. Moreover, it is clear from the definition that a graph is partitionable if and only if its complement is partitionable. Hence, for $\alpha \geq 2$ and $\omega \geq 2$, both $C_{\alpha\omega+1}^{\omega-1}$ and $\overline{C_{\alpha\omega+1}^{\omega-1}}$ provide infinite families of graphs whose associated polynomials are nonnegative but not sos. We remark that the clique number of $C_{\alpha\omega+1}^{\omega-1}$ is $\omega$ and the clique number of $\overline{C_{\alpha\omega+1}^{\omega-1}}$ is $\alpha$, and therefore it is immediate to explicitly write down the polynomial $p_G(x)$ when $G$ is the graph $C_{\alpha\omega+1}^{\omega-1}$ or $\overline{C_{\alpha\omega+1}^{\omega-1}}$.

\subsubsection{Paley graphs}

Paley graphs are graphs constructed from the elements of certain finite fields by connecting pairs of elements that differ by a quadratic residue. More precisely, for a prime number\footnote{Paley graphs are more generally defined for prime powers. Here we are only interested in Paley graphs with a prime number of vertices.} $q$ with $q \equiv 1~\pmod 4$, the Paley graph $P_q$ is the graph with vertices the elements of the finite field $F_q$, which can be represented by integers $0, 1, \dots, q-1$, and an edge between two vertices $x$ and $y$ if and only if $x- y = a^2$ for some nonzero element $a \in F_q$. Paley graphs have been extensively studied due to their interesting symmetry properties. In particular, Paley graphs are self-complementary and edge-transitive (see, e.g.,~\cite{BelaBol} for these and other properties of Paley graphs). See Figure~\ref{fig:paley_graphs} for two examples of Paley graphs.
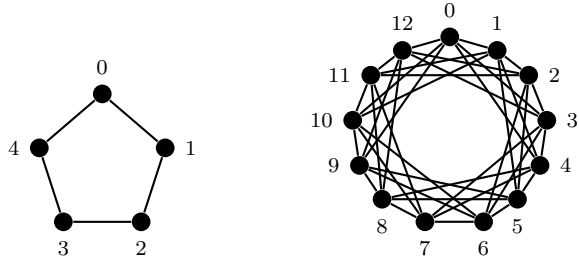
\begin{figure}[h]
\begin{center}
\begin{tikzpicture}[scale=0.22]

\node[label=above:{\scriptsize{$0$}}, inner sep=2.5pt, fill=black, circle] at (0,4.5)(v0){}; 
\node[label=right:{\scriptsize{$1$}}, inner sep=2.5pt, fill=black, circle] at (3.804226, 1.2361)(v1){}; 
\node[label=below:{\scriptsize{$2$}}, inner sep=2.5pt, fill=black, circle] at (2.351141, -3.2361)(v2){}; 
\node[label=below:{\scriptsize{$3$}}, inner sep=2.5pt, fill=black, circle] at (-2.351141, -3.2361)(v3){}; 
\node[label=left:{\scriptsize{$4$}}, inner sep=2.5pt, fill=black, circle] at (-3.804226, 1.2361)(v4){};

\draw[black, thick] (v0) -- (v1);
\draw[black, thick] (v0) -- (v4);
\draw[black, thick] (v1) -- (v2);
\draw[black, thick] (v2) -- (v3);
\draw[black, thick] (v3) -- (v4);

\end{tikzpicture}
\hspace{1cm}
\begin{tikzpicture}[scale=0.3]

\node[label=above:{\scriptsize{$0$}}, inner sep=2.5pt, fill=black, circle] at (1.5,8.2)(v0){}; 
\node[label=above:{\scriptsize{$1$}}, inner sep=2.5pt, fill=black, circle] at (3.6, 7.6)(v1){}; 
\node[label=right:{\scriptsize{$2$}}, inner sep=2.5pt, fill=black, circle] at (5, 6.5)(v2){}; 
\node[label=right:{\scriptsize{$3$}}, inner sep=2.5pt, fill=black, circle] at (5.8, 4.5)(v3){}; 
\node[label=right:{\scriptsize{$4$}}, inner sep=2.5pt, fill=black, circle] at (5.5, 2.5)(v4){};
\node[label=below:{\scriptsize{$5$}}, inner sep=2.5pt, fill=black, circle] at (4.5,1)(v5){};

\node[label=below:{\scriptsize{$6$}}, inner sep=2.5pt, fill=black, circle] at (3, 0)(v6){};
\node[label=below:{\scriptsize{$7$}}, inner sep=2.5pt, fill=black, circle] at (0.4, 0)(v7){};

\node[label=below:{\scriptsize{$8$}}, inner sep=2.5pt, fill=black, circle] at (-1.5, 1)(v8){}; 
\node[label=left:{\scriptsize{$9$}}, inner sep=2.5pt, fill=black, circle] at (-2.5, 2.5)(v9){};
\node[label=left:{\scriptsize{$10$}}, inner sep=2.5pt, fill=black, circle] at (-2.8, 4.5)(v10){}; 
\node[label=left:{\scriptsize{$11$}}, inner sep=2.5pt, fill=black, circle] at (-2, 6.5)(v11){}; 
\node[label=above:{\scriptsize{$12$}}, inner sep=2.5pt, fill=black, circle] at (-0.6, 7.6)(v12){};

\draw[black, thick] (v0) -- (v1);
\draw[black, thick] (v1) -- (v2);
\draw[black, thick] (v2) -- (v3);
\draw[black, thick] (v3) -- (v4);
\draw[black, thick] (v4) -- (v5);
\draw[black, thick] (v5) -- (v6);
\draw[black, thick] (v6) -- (v7);
\draw[black, thick] (v7) -- (v8);
\draw[black, thick] (v8) -- (v9);
\draw[black, thick] (v9) -- (v10);
\draw[black, thick] (v10) -- (v11);
\draw[black, thick] (v11) -- (v12);
\draw[black, thick] (v12) -- (v0);

\draw[black, thick] (v0) -- (v3);
\draw[black, thick] (v0) -- (v4);
\draw[black, thick] (v0) -- (v9);
\draw[black, thick] (v0) -- (v10);

\draw[black, thick] (v1) -- (v4);
\draw[black, thick] (v1) -- (v5);
\draw[black, thick] (v1) -- (v10);
\draw[black, thick] (v1) -- (v11);

\draw[black, thick] (v2) -- (v5);
\draw[black, thick] (v2) -- (v6);
\draw[black, thick] (v2) -- (v11);
\draw[black, thick] (v2) -- (v12);

\draw[black, thick] (v3) -- (v6);
\draw[black, thick] (v3) -- (v7);
\draw[black, thick] (v3) -- (v12);

\draw[black, thick] (v4) -- (v7);
\draw[black, thick] (v4) -- (v8);
\draw[black, thick] (v5) -- (v8);
\draw[black, thick] (v5) -- (v9);
\draw[black, thick] (v6) -- (v9);
\draw[black, thick] (v6) -- (v10);
\draw[black, thick] (v7) -- (v10);
\draw[black, thick] (v7) -- (v11);
\draw[black, thick] (v8) -- (v11);
\draw[black, thick] (v8) -- (v12);
\draw[black, thick] (v9) -- (v12);

\end{tikzpicture}
\end{center}
\vspace{-0.6cm}
\caption{The Paley graphs on $5$ and $13$ vertices}
\label{fig:paley_graphs}
\end{figure}

Let $G$ be an edge-transitive graph and let $\lambda_{\text{max}}$ and $\lambda_{\text{min}}$ respectively denote the largest and smallest eigenvalues of its adjacency matrix. It is shown in \cite[Corollary 5.3]{GRSS} that $$\vartheta'(\overline{G}) = \vartheta(\overline{G}) = 1 - \frac{\lambda_{\text{max}}}{\lambda_{\text{min}}}.$$
Since the Paley graph $P_q$ is edge-transitive and the distinct eigenvalues of its adjacency matrix are known to be $\frac{1}{2}(q-1)$, $\frac{1}{2}(\sqrt{q} - 1)$ (see, e.g., \cite[Proposition 9.1.1]{BroHae}), and $\frac{1}{2}(- \sqrt{q} - 1)$, it follows that $\vartheta'(\overline{P_q}) = \sqrt{q}$. In particular, for a prime number $q \equiv 1 \pmod 4$, as $\sqrt{q}$ is not an integer, we have $\omega(P_q) < \vartheta'(\overline{P_q})$. Therefore, by Corollary~\ref{cor:pgx_nonnegative_sos}, the polynomial $p_{P_q}(x)$ is nonnegative but not sos. Thus, for primes $q \equiv 1 \pmod 4$, the family of Paley graphs $P_q$ yields another infinite family of polynomials that are nonnegative but not sos.\footnote{Paley graphs on a prime number of vertices form a subclass of the so-called ``circulant graphs''. It is known that for a circulant graph $G$ on a prime number of vertices, we have $\omega(G) < \vartheta'(\overline{G})$ (see \cite{BFL}). Hence, more generally, circulant graphs yield an infinite family of polynomials that are nonnegative but not sos.}

Although the clique number of Paley graphs is in general not known, an upper bound was recently given in \cite{HanPet} for any prime number $q$:
$$\omega(P_q) \leq \frac{\sqrt{2q-1}+1}{2}.$$
Note that the inequality $\frac{\sqrt{2q-1}+1}{2} < \sqrt{q}$ holds for any integer $q > 1$. Therefore, for every prime number $q \equiv 1 \pmod 4$ and for every $k \in \R$ with $\Big \lfloor \frac{\sqrt{2q-1}+1}{2} \Big \rfloor \leq k < \sqrt{q}$, the polynomial $p_{P_q,k} (x)$ as defined in~\eqref{eqn:pgkx} is nonnegative but not sos.

\subsubsection{Mycielski graphs}



For a graph $G$ with vertex set $V(G) = \{v_1, \dots, v_n\}$, the \emph{Mycielskian} of $G$, denoted by $M(G)$, is the graph obtained from $G$ by adding $n+1$ new vertices $u_1, \dots, u_n, w$, and for $1 \leq i \leq n$, making $u_i$ adjacent to the neighbors of $v_i$ and to $w$. The sequence of graphs $M_2, M_3, M_4, \dots$ obtained by starting with the one-edge graph $M_2$ and applying the Mycielskian operation $M_{k+1} = M(M_{k})$ repeatedly for $k\geq 2$ is called the Mycielski graphs. It is well known that for every $k=2,3,\dots$, we have $\omega(M_k) = 2$ and $\chi(M_k) = k$ \cite{Mycielski}. In other words, the Mycielskian operation preserves the property of having clique number equal to 2 but increases the chromatic number. The first few graphs in this sequence are the one-edge graph $M_2$, the $5$-vertex cycle graph $M_3 = C_5$, and the Gr\"otzsch graph $M_4$. See Figure~\ref{fig:mycielski_graphs} for the Mycielski graphs $M_3$ and $M_4$.

\begin{figure}[h]
\begin{center}
\begin{tikzpicture}[scale=0.22]

\node[inner sep=2.5pt, fill=black, circle] at (0,4.5)(v0){}; 
\node[inner sep=2.5pt, fill=black, circle] at (3.804226, 1.2361)(v1){}; 
\node[inner sep=2.5pt, fill=black, circle] at (2.351141, -3.2361)(v2){}; 
\node[inner sep=2.5pt, fill=black, circle] at (-2.351141, -3.2361)(v3){}; 
\node[inner sep=2.5pt, fill=black, circle] at (-3.804226, 1.2361)(v4){};

\draw[black, thick] (v0) -- (v1);
\draw[black, thick] (v0) -- (v4);
\draw[black, thick] (v1) -- (v2);
\draw[black, thick] (v2) -- (v3);
\draw[black, thick] (v3) -- (v4);

\end{tikzpicture}
\hspace{1.3cm}
\begin{tikzpicture}[scale=0.6]
\draw[every node/.style={inner sep=2.5pt,fill=black, circle}]
\foreach \x in {0,...,4}{(90+72*\x:1) node(x\x){}  (90+72*\x:2) node(y\x){}}
(0,0) node{};

\draw[black, thick] (y0) -- (y1);
\draw[black, thick] (y1) -- (y2);
\draw[black, thick] (y2) -- (y3);
\draw[black, thick] (y3) -- (y4);
\draw[black, thick] (y4) -- (y0);

\draw[black, thick] (y0) -- (x1);
\draw[black, thick] (x1) -- (y2);
\draw[black, thick] (y2) -- (x3);
\draw[black, thick] (x3) -- (y4);
\draw[black, thick] (y4) -- (x0);
\draw[black, thick] (x0) -- (y1);
\draw[black, thick] (y1) -- (x2);
\draw[black, thick] (x2) -- (y3);
\draw[black, thick] (y3) -- (x4);
\draw[black, thick] (x4) -- (y0);

\draw[black, thick] (x0) -- (0,0);
\draw[black, thick] (x1) -- (0,0);
\draw[black, thick] (x2) -- (0,0);
\draw[black, thick] (x3) -- (0,0);
\draw[black, thick] (x4) -- (0,0);

\end{tikzpicture}
\end{center}
\vspace{-0.5cm}
\caption{The Mycielski graphs $M_3$ and $M_4$}
\label{fig:mycielski_graphs}
\end{figure}
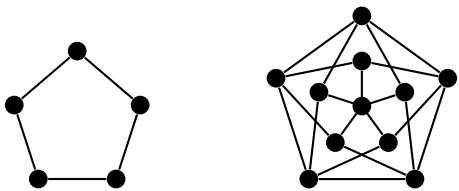

Since $\omega(M_3) < \vartheta'(\overline{M_3})$ and since $M_{k+1}$ is obtained from $M_k$ by adding vertices and edges without increasing the clique number, it follows from the arguments in Section \ref{subsec:add_vertices} below that for $k \geq 3$, we have $\omega(M_k) < \vartheta'(\overline{M_k})$.\footnote{The inequality $\omega(M_k) < \vartheta'(\overline{M_k})$ for $k \geq 3$ also follows from a lemma of Laurent and Vargas \cite[Lemma 5.6]{LV2} which states that for graphs $G$ with $\omega(G) \leq 2$, $\omega(G) = \vartheta(\overline{G})$ if and only if $\omega(G) = \chi(G)$. Since $2 = \omega(M_k) < \chi(M_k) = k$ for $k \geq 3$, the result follows.} Therefore, by Corollary~\ref{cor:pgx_nonnegative_sos}, the polynomial $p_{M_k}(x)$ is nonnegative but not sos. Thus, for $k \geq 3$, the family of Mycielski graphs $M_k$ yields another infinite family of polynomials that are nonnegative but not sos and that can explicitly be written down (since $\omega(M_k) = 2$).

\subsection{Graph-theoretic operations that preserve the property of being not sos} \label{subsec:graph_operations}

\subsubsection{Adding edges and vertices without increasing the clique number} 
\label{subsec:add_vertices}

It is easy to observe that adding edges to a graph without increasing its clique number preserves the property of being not sos. Indeed, if $H$ is the graph obtained from a graph $G$ by adding the edge $ij$ such that $\omega(H) = \omega(G)$, then from the definition in \eqref{eqn:our_polynomial}, we have $p_G(x) = p_H(x) + 2 \omega(G) x_i^2x_j^2$. This observation also appears in \cite[Lemma 4.9]{LV2} (see also \cite{LV1}) with a slightly different language.

The following lemma, whose proof is simple and thus omitted, also allows us to add new vertices with arbitrary adjacencies.
\begin{lemma}
Let $G$ be a graph with vertex set $V(G) = \{v_1, \dots, v_n\}$ and let $H$ be an induced subgraph of $G$ with $V(H) = \{v_1, \dots, v_k\}$ (i.e., the graph $H$ is obtained from $G$ by deleting the vertices $v_{k+1}, \dots, v_n$). Then, 
$$p_G(x_1, \dots, x_k, 0, \dots, 0) = p_H(x_1, \dots, x_k)$$ if and only if $\omega(G) = \omega(H)$.
\end{lemma}

Using the ``if direction'' of this lemma, we can generate many more nonnegative polynomials that are not sos starting from existing ones. Let $H$ be a graph with $\omega(H) < \vartheta'(\overline{H})$ (such as the graphs described in Section~\ref{subsec:graphs}). By Corollary~\ref{cor:pgx_nonnegative_sos}, the polynomial $p_H(x)$ is nonnegative but not sos. Let $G$ be a graph obtained from $H$ by adding new vertices with arbitrary adjacencies such that $\omega(G) = \omega(H)$. Then, the polynomial $p_{G}(x)$ is nonnegative but not sos. Indeed, since $\omega(G) = \omega(H)$, by setting the variables that correspond to vertices in $V(G) \setminus V(H)$ to zero, we obtain the polynomial $p_H(x)$. Hence, if $p_{G}(x)$ was sos, then $p_H(x)$ would be sos since it is obtained from an sos polynomial by setting some variables to zero. 

As an example of this operation, let $H$ be the graph $C_5$ and $G$ be the graph in Figure \ref{fig:set_to_zero_example} (right), which is arbitrarily constructed by adding vertices to $H$ without increasing the clique number. We have $p_{G}(x_1, \dots, x_5, 0, \dots, 0) = p_H(x_1, \dots, x_5)$, and since $p_H(x_1, \dots, x_5)$ is not sos, the polynomial $p_{G}(x_1, \dots, x_{14})$ is not sos.
\begin{figure}[h]
\begin{center}
\begin{tikzpicture}[scale=0.22]

\node[label=above:{\scriptsize{$v_1$}}, inner sep=2.5pt, fill=black, circle] at (0,4.5)(v1){}; 
\node[label=right:{\scriptsize{$v_2$}}, inner sep=2.5pt, fill=black, circle] at (3.804226, 1.2361)(v2){}; 
\node[label=below:{\scriptsize{$v_3$}}, inner sep=2.5pt, fill=black, circle] at (2.351141, -3.2361)(v3){}; 
\node[label=below:{\scriptsize{$v_4$}}, inner sep=2.5pt, fill=black, circle] at (-2.351141, -3.2361)(v4){}; 
\node[label=left:{\scriptsize{$v_5$}}, inner sep=2.5pt, fill=black, circle] at (-3.804226, 1.2361)(v5){};

\draw[black, thick] (v1) -- (v2);
\draw[black, thick] (v1) -- (v5);
\draw[black, thick] (v2) -- (v3);
\draw[black, thick] (v3) -- (v4);
\draw[black, thick] (v4) -- (v5);

\node at (0, -8.5) {\small{$H$}};

\end{tikzpicture}
\hspace{2cm}
\begin{tikzpicture}[scale=0.22]

\node[label=above:{\scriptsize{$v_1$}}, inner sep=2.5pt, fill=black, circle] at (0,4.5)(v1){}; 
\node[label=left:{\scriptsize{$v_2$}}, inner sep=2.5pt, fill=black, circle] at (3.804226, 1.2361)(v2){}; 
\node[label=below:{\scriptsize{$v_3$}}, inner sep=2.5pt, fill=black, circle] at (2.351141, -3.2361)(v3){}; 
\node[label=below:{\scriptsize{$v_4$}}, inner sep=2.5pt, fill=black, circle] at (-2.351141, -3.2361)(v4){}; 
\node[label=above:{\scriptsize{$v_5$}}, inner sep=2.5pt, fill=black, circle] at (-3.804226, 1.2361)(v5){};
\node[label=left:{\scriptsize{$v_6$}}, inner sep=2.5pt, fill=black, circle] at (-6, 3)(v6){}; 
\node[label=left:{\scriptsize{$v_7$}}, inner sep=2.5pt, fill=black, circle] at (-6, 1)(v7){}; 
\node[label=left:{\scriptsize{$v_8$}}, inner sep=2.5pt, fill=black, circle] at (-6, -1)(v8){};
\node[label=left:{\scriptsize{$v_9$}}, inner sep=2.5pt, fill=black, circle] at (-4.2, -6)(v9){}; 
\node[label=left:{\scriptsize{$v_{10}$}}, inner sep=2.5pt, fill=black, circle] at (-6.4, -3.3)(v10){}; 
\node[label=above:{\scriptsize{$v_{11}$}}, inner sep=2.5pt, fill=black, circle] at (6, 3)(v11){}; 
\node[label=below:{\scriptsize{$v_{12}$}}, inner sep=2.5pt, fill=black, circle] at (6, -1.5)(v12){}; 
\node[label=below:{\scriptsize{$v_{13}$}}, inner sep=2.5pt, fill=black, circle] at (10, -1.5)(v13){}; 
\node[label=above:{\scriptsize{$v_{14}$}}, inner sep=2.5pt, fill=black, circle] at (10, 3)(v14){}; 

\draw[black, thick] (v1) -- (v2);
\draw[black, thick] (v1) -- (v5);
\draw[black, thick] (v2) -- (v3);
\draw[black, thick] (v3) -- (v4);
\draw[black, thick] (v4) -- (v5);
\draw[black, thick] (v5) -- (v6);
\draw[black, thick] (v5) -- (v7);
\draw[black, thick] (v5) -- (v8);
\draw[black, thick] (v4) -- (v9);
\draw[black, thick] (v9) -- (v10);
\draw[black, thick] (v1) -- (v11);
\draw[black, thick] (v3) -- (v12);
\draw[black, thick] (v11) -- (v12);
\draw[black, thick] (v12) -- (v13);
\draw[black, thick] (v13) -- (v14);
\draw[black, thick] (v11) -- (v14);

\node at (0, -8.5) {\small{$G$}};

\end{tikzpicture}
\end{center}
\vspace{-0.5cm}
\caption{The graph $G$ is obtained from the graph $H$ by adding new vertices with arbitrary adjacencies such that $\omega(G) = \omega(H)$, making $p_G$ inherit the property of being not sos from $p_H$}
\label{fig:set_to_zero_example}
\end{figure}
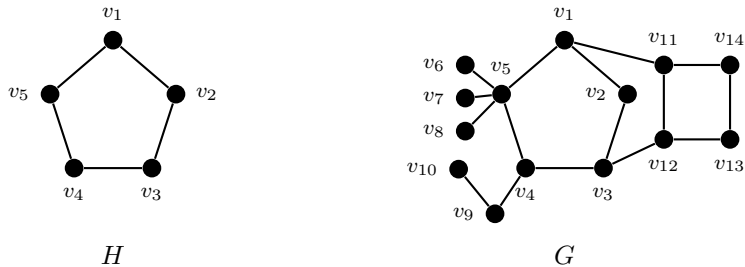

We note that for a graph $H$ for which $p_H(x)$ is not sos, if $G$ is obtained from $H$ by adding new vertices in a way that $\omega(G) > \omega(H)$, then $p_G(x)$ might become sos. See the last two sentences of Remark~\ref{rem:when_sos} for an example.

\subsubsection{Graph join}

The \emph{join} of two graphs $G_1$ and $G_2$, with disjoint vertex sets, is the graph obtained by connecting every vertex of $G_1$ to all vertices of $G_2$. See Figure \ref{fig:join_G1_G2} for an example. It is easy to see that if $G$ is the join of $G_1$ and $G_2$, then $\omega(G) = \omega(G_1) + \omega(G_2)$. It is also true that $\vartheta'(\overline{G}) = \vartheta'(\overline{G_1}) + \vartheta'(\overline{G_2})$ (see Theorem~4.1 in \cite{BFL}). Hence, if we start with two graphs $G_1$ and $G_2$ such that $\omega(G_i) < \vartheta'(\overline{G}_i)$ for either $i=1$ or $i=2$, then the join $G$ of $G_1$ and $G_2$ satisfies $\omega(G) < \vartheta'(\overline{G})$. Therefore, for every graph $G$ obtained this way, the polynomial $p_{G}(x)$ is a nonnegative polynomial that is not sos. As an example, consider the graphs $G_1,G_2$, and their join $G$, given in Figure \ref{fig:join_G1_G2}. We have 
$$4 = \omega(G_1) + \omega(G_2) = \omega(G) < \vartheta'(\overline{G}) = \vartheta'(\overline{G_1}) + \vartheta'(\overline{G_2}) = \sqrt{5} + 2 ,$$
and thus the polynomial $p_G(x)$ is nonnegative but not sos. 
\begin{figure}[h]
\begin{center}
\begin{tikzpicture}[scale=0.22]

\node at (0, -6.5) {\small{$G_1$}};

\node[inner sep=2.5pt, fill=black, circle] at (0,4.5)(v1){}; 
\node[inner sep=2.5pt, fill=black, circle] at (3.804226, 1.2361)(v2){}; 
\node[inner sep=2.5pt, fill=black, circle] at (2.351141, -3.2361)(v3){}; 
\node[inner sep=2.5pt, fill=black, circle] at (-2.351141, -3.2361)(v4){}; 
\node[inner sep=2.5pt, fill=black, circle] at (-3.804226, 1.2361)(v5){};

\draw[black, thick] (v1) -- (v2);
\draw[black, thick] (v1) -- (v5);
\draw[black, thick] (v2) -- (v3);
\draw[black, thick] (v3) -- (v4);
\draw[black, thick] (v4) -- (v5);

\end{tikzpicture}
\hspace{0.5cm}
\begin{tikzpicture}[scale=0.24]

\node at (0, -6.5) {\small{$G_2$}};

\node[inner sep=2.5pt, fill=black, circle] at (0,2)(v6){}; 
\node[inner sep=2.5pt, fill=black, circle] at (0, -2)(v7){}; 

\draw[black, thick] (v6) -- (v7);

\end{tikzpicture}
\hspace{3cm}
\begin{tikzpicture}[scale=0.22]

\node at (2.8, -6.5) {\small{$G$}};

\node[inner sep=2.5pt, fill=black, circle] at (0,4.5)(v1){}; 
\node[inner sep=2.5pt, fill=black, circle] at (3.804226, 1.2361)(v2){}; 
\node[inner sep=2.5pt, fill=black, circle] at (2.351141, -3.2361)(v3){}; 
\node[inner sep=2.5pt, fill=black, circle] at (-2.351141, -3.2361)(v4){}; 
\node[inner sep=2.5pt, fill=black, circle] at (-3.804226, 1.2361)(v5){};
\node[inner sep=2.5pt, fill=black, circle] at (8,2)(v6){}; 
\node[inner sep=2.5pt, fill=black, circle] at (8,-2)(v7){}; 

\draw[black, thick](v1) edge [bend right=20] (v7);
\draw[black, thick](v5) edge [bend left=10] (v6);
\draw[black, thick](v4) edge [bend right=10] (v6);
\draw[black, thick](v4) edge [bend right=35] (v7);

\draw[black, thick] (v1) -- (v2);
\draw[black, thick] (v1) -- (v5);
\draw[black, thick] (v2) -- (v3);
\draw[black, thick] (v3) -- (v4);
\draw[black, thick] (v4) -- (v5);
\draw[black, thick] (v6) -- (v7);

\draw[black, thick] (v6) -- (v1);
\draw[black, thick] (v6) -- (v2);
\draw[black, thick] (v6) -- (v3);
\draw[black, thick] (v7) -- (v2);
\draw[black, thick] (v7) -- (v3);
\draw[black, thick] (v7) -- (v5);

\end{tikzpicture}
\end{center}
\vspace{-0.6cm}
\caption{Since $p_{G_1}$ is not sos, the polynomial $p_G$ associated with the join $G$ of $G_1$ and $G_2$ is not sos}
\label{fig:join_G1_G2}
\end{figure}
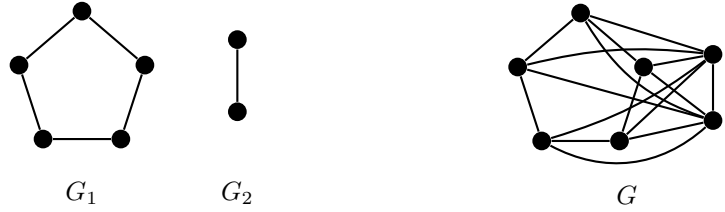

\subsubsection{Graph strong product}

The \emph{strong product} of two graphs $G_1$ and $G_2$, denoted by $G_1 \boxtimes G_2$, is the graph with vertex set $V(G_1 \boxtimes G_2) = V(G_1) \times V(G_2)$, where two vertices $(a_1, a_2)$ and $(b_1, b_2)$ are adjacent if and only if
\begin{itemize}
\itemsep0em
\item $a_1 = b_1$ and $a_2b_2 \in E(G_2)$, or
\item $a_1b_1 \in E(G_1)$ and $a_2 = b_2$, or
\item $a_1b_1 \in E(G_1)$ and $a_2b_2 \in E(G_2)$.
\end{itemize}
It is not difficult to see that\footnote{We would like to warn the reader that the equality $\alpha(G_1 \boxtimes G_2) = \alpha(G_1) \alpha(G_2)$ is not necessarily true. Although $\alpha(G_1 \boxtimes G_2) \geq \alpha(G_1) \alpha(G_2)$ is always true, there are graphs $G_1, G_2$ that make the inequality strict, e.g., $G_1=G_2=C_5$.} $\omega(G_1 \boxtimes G_2) = \omega(G_1) \omega(G_2)$ holds for any two graphs $G_1,G_2$ (see, e.g., Lemma 3.1 in \cite{DGH}). It is also true that\footnote{We would like to warn the reader that the equality $\vartheta'(G_1 \boxtimes G_2) = \vartheta'(G_1) \vartheta'(G_2)$ is not necessarily true. Although $\vartheta'(G_1 \boxtimes G_2) \geq \vartheta'(G_1) \vartheta'(G_2)$ is always true, there are graphs $G_1, G_2$ that make the inequality strict. See \cite{CMRSSW} for details.} $\vartheta'(\overline{G_1 \boxtimes G_2}) = \vartheta'(\overline{G_1}) \vartheta'(\overline{G_2})$ for any two graphs $G_1,G_2$, see, e.g., Theorem 25\footnote{Theorem 25 in \cite{CMRSSW} shows that $\vartheta'(G_1 \ast G_2) = \vartheta'(G_1) \vartheta'(G_2)$, where $G_1 \ast G_2$ denotes the \emph{disjunctive product} of $G_1$ and $G_2$ (see \cite{CMRSSW} for the definition). We then have $\vartheta'(\overline{G_1 \boxtimes G_2}) = \vartheta'(\overline{G_1} \ast \overline{G_2}) = \vartheta'(\overline{G_1}) \vartheta'(\overline{G_2}).$} in \cite{CMRSSW}. (In fact, for our purposes, the inequalities $\omega(G_1 \boxtimes G_2) \leq \omega(G_1) \omega(G_2)$ and $\vartheta'(\overline{G_1 \boxtimes G_2}) \geq \vartheta'(\overline{G_1}) \vartheta'(\overline{G_2})$ are enough.) Hence, if we start with two graphs $G_1$ and $G_2$ such that $\omega(G_i) < \vartheta'(\overline{G}_i)$ for either $i=1$ or $i=2$, then 
$$\omega(G_1 \boxtimes G_2) = \omega(G_1) \omega(G_2) < \vartheta'(\overline{G_1}) \vartheta'(\overline{G_2}) = \vartheta'(\overline{G_1 \boxtimes G_2}),$$
and so the graph $G_1 \boxtimes G_2$ satisfies $\omega(G_1 \boxtimes G_2) < \vartheta'(\overline{G_1 \boxtimes G_2})$. Therefore, for every graph $G$ obtained this way, the polynomial $p_{G}(x)$ is a nonnegative polynomial that is not sos. As an example, consider the graphs $G_1,G_2$, and their strong product $G$, given in Figure \ref{fig:product_G1_G2}. We have 
$$4 = \omega(G_1) \omega(G_2) = \omega(G) < \vartheta'(\overline{G}) = \vartheta'(\overline{G_1}) \vartheta'(\overline{G_2}) = 2\sqrt{5},$$
and thus the polynomial $p_G(x)$ is nonnegative but not sos.
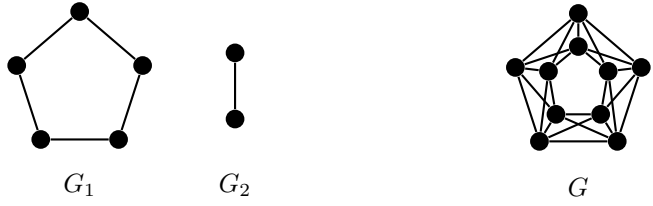
\begin{figure}[h]
\begin{center}
\begin{tikzpicture}[scale=0.22]

\node at (0, -6) {\small{$G_1$}};

\node[inner sep=2.5pt, fill=black, circle] at (0,4.5)(v1){}; 
\node[inner sep=2.5pt, fill=black, circle] at (3.804226, 1.2361)(v2){}; 
\node[inner sep=2.5pt, fill=black, circle] at (2.351141, -3.2361)(v3){}; 
\node[inner sep=2.5pt, fill=black, circle] at (-2.351141, -3.2361)(v4){}; 
\node[inner sep=2.5pt, fill=black, circle] at (-3.804226, 1.2361)(v5){};

\draw[black, thick] (v1) -- (v2);
\draw[black, thick] (v1) -- (v5);
\draw[black, thick] (v2) -- (v3);
\draw[black, thick] (v3) -- (v4);
\draw[black, thick] (v4) -- (v5);

\end{tikzpicture}
\hspace{0.5cm}
\begin{tikzpicture}[scale=0.22]

\node at (0, -6) {\small{$G_2$}};

\node[inner sep=2.5pt, fill=black, circle] at (0,2)(v6){}; 
\node[inner sep=2.5pt, fill=black, circle] at (0, -2)(v7){}; 

\draw[black, thick] (v6) -- (v7);

\end{tikzpicture}
\hspace{3cm}
\begin{tikzpicture}[scale=0.22]

\node at (0, -6) {\small{$G$}};

\node[inner sep=2.5pt, fill=black, circle] at (0,4.5)(v1){}; 
\node[inner sep=2.5pt, fill=black, circle] at (3.804226, 1.2361)(v2){}; 
\node[inner sep=2.5pt, fill=black, circle] at (2.351141, -3.2361)(v3){}; 
\node[inner sep=2.5pt, fill=black, circle] at (-2.351141, -3.2361)(v4){}; 
\node[inner sep=2.5pt, fill=black, circle] at (-3.804226, 1.2361)(v5){};

\node[inner sep=2.5pt, fill=black, circle] at (0,2.5)(v6){}; 
\node[inner sep=2.5pt, fill=black, circle] at (1.804226, 1)(v7){}; 
\node[inner sep=2.5pt, fill=black, circle] at (1.351141, -1.6)(v8){}; 
\node[inner sep=2.5pt, fill=black, circle] at (-1.351141, -1.6)(v9){}; 
\node[inner sep=2.5pt, fill=black, circle] at (-1.804226, 1)(v10){};

\draw[black, thick] (v1) -- (v2);
\draw[black, thick] (v2) -- (v3);
\draw[black, thick] (v3) -- (v4);
\draw[black, thick] (v4) -- (v5);
\draw[black, thick] (v5) -- (v1);

\draw[black, thick] (v6) -- (v7);
\draw[black, thick] (v7) -- (v8);
\draw[black, thick] (v8) -- (v9);
\draw[black, thick] (v9) -- (v10);
\draw[black, thick] (v10) -- (v6);

\draw[black, thick] (v1) -- (v6);
\draw[black, thick] (v2) -- (v7);
\draw[black, thick] (v3) -- (v8);
\draw[black, thick] (v4) -- (v9);
\draw[black, thick] (v5) -- (v10);

\draw[black, thick] (v1) -- (v7);
\draw[black, thick] (v2) -- (v6);

\draw[black, thick] (v2) -- (v8);
\draw[black, thick] (v3) -- (v7);

\draw[black, thick] (v3) -- (v9);
\draw[black, thick] (v4) -- (v8);

\draw[black, thick] (v4) -- (v10);
\draw[black, thick] (v5) -- (v9);

\draw[black, thick] (v5) -- (v6);
\draw[black, thick] (v1) -- (v10);

\end{tikzpicture}
\end{center}
\vspace{-0.6cm}
\caption{Since $p_{G_1}$ is not sos, the polynomial $p_G$ associated with the strong product $G$ of $G_1$ and $G_2$ is not sos}
\label{fig:product_G1_G2}
\end{figure}

\begin{remark}
Although not explicitly mentioned, it follows from Theorem 3 of a work of Dickinson and de Zeeuw~\cite{DdZ}, which was recently brought to our attention, that for every graph $G$ with at least one edge whose complement is connected, $\alpha$-critical, and ``$\alpha$-covered'' (see~\cite{DdZ} for definitions), the nonnegative polynomial $p_G(x)$ in~\eqref{eqn:our_polynomial} would not be sos. However, these three conditions are quite restrictive for our purposes. In particular, using the operations presented above, it is easy to produce many examples of graphs $G$ for which $p_G(x)$ is not sos and such that any one of three conditions of \cite[Theorem 3]{DdZ} is violated.
\end{remark}

\subsection{Random graphs} \label{subsec:random_graphs}

For a positive integer $n$ and a real number $p \in (0,1)$, an Erd\H{o}s-R\'enyi random graph $G_{n,p}$, introduced in~\cite{ErdRen}, is a graph on $n$ vertices where each vertex is adjacent to each other vertex with probability $p$, independent of all other choices. In this section, we consider the polynomial $p_{G_{n,p}}(x)$ (as defined in~\eqref{eqn:our_polynomial}) associated with an Erd\H{o}s-R\'enyi random graph $G_{n,p}$. For a constant $p \in (0,1)$, we would like to understand the probability that the nonnegative polynomial $p_{G_{n,p}}(x)$ is not sos as $n$ tends to infinity.

\begin{lemma}
Let $p \in (0,1)$ be fixed. Then, for every integer $n$ that satisfies $1 \leq 2 \log_{1/p} n  \leq n$, we have
$$\textup{Pr} \Big(\omega(G_{n,p}) < 2 \log_{1/p} n \Big) \geq 1 - n \: \Bigg( \frac{e}{2 \log_{1/p} n} \Bigg)^{2 \log_{1/p} n}.$$
\label{lem:clique_random}
\end{lemma}

\begin{proof}
Let $X_r$ denote the number of cliques of size $r$ in $G_{n,p}$. Note that
\begin{equation}
\begin{aligned}
\hspace{2cm}
\textup{Pr} \Big(\omega(G_{n,p}) < r \Big) = \textup{Pr} (X_r = 0) = 1 - \textup{Pr} (X_r \geq 1) \geq 1 - \textup{E}[X_r].
\end{aligned}
\label{eqn:prob_Xr}
\end{equation}
For any $n \geq r \geq 1$, we have
$$\textup{E}[X_r] = {n \choose r} p^{{r \choose 2}} \leq \Bigg(\frac{ne}{r} \Bigg)^r \: p^{\frac{r(r-1)}{2}},$$
where the inequality follows since ${n \choose r} \leq \big(\frac{ne}{r} \big)^r$. If we set $r = 2 \log_{1/p} n$, we obtain
$$\textup{E}[X_r] \leq  n \: \Bigg( \frac{e}{2 \log_{1/p} n} \Bigg)^{2 \log_{1/p} n},$$
which combined with \eqref{eqn:prob_Xr} proves the claim.
\end{proof}

\begin{lemma}[\cite{Coja-Oghlan}]
Let $p \in (0,1)$ be fixed. Then, there exists a constant $\lambda > 0$ such that for large enough $n$, we have
$$\textup{Pr} \Bigg( \vartheta'(\overline{G_{n,p}}) > \frac{1}{2(\lambda+4)} \sqrt{\frac{np}{1-p}}\Bigg) \geq 1 - e^{-n}.$$
\label{lem:schrijver_random}
\end{lemma}

The following theorem is then a direct consequence of Lemma~\ref{lem:nonneg_if_clique}, and the union bound applied to Lemma~\ref{lem:clique_random} and Lemma~\ref{lem:schrijver_random}. Recall the definition of the polynomial $p_{G,k}(x)$ from \eqref{eqn:pgkx}.

\begin{theorem}
Let $p \in (0,1)$ be fixed. Then, there exists a constant $\lambda > 0$ such that for large enough $n$ and for any $k$ satisfying $2 \log_{1/p} n \leq k \leq \frac{1}{2(\lambda+4)} \sqrt{\frac{np}{1-p}}$, we have
\begin{equation}
\hspace{1cm}
\textup{Pr} \Big(p_{G_{n,p},k}(x) \text{ is nonnegative but not sos} \Big) \geq 1 - e^{-n} - n \: \Bigg( \frac{e}{2 \log_{1/p} n} \Bigg)^{2 \log_{1/p} n}.
\label{eq:probability_not_sos}
\end{equation}
\label{thm:prob_random}
\end{theorem}

Notice that the range of allowed values for $k$ in Theorem~\ref{thm:prob_random} gets larger as $n$ increases, and that the right hand side in~\eqref{eq:probability_not_sos} tends to 1 as $n$ tends to infinity. Note also that for any $n,p,k$, the probability that the (nonnegative) polynomial $p_{G_{n,p}}(x)$ is not sos is greater than or equal to the probability that the polynomial $p_{G_{n,p},k}(x)$ is nonnegative but not sos. Therefore, for a fixed constant $p \in (0,1)$ and for large enough $n$, the polynomial $p_{G_{n,p}}(x)$ is nonnegative but not sos with high probability. In some sense, Theorem~\ref{thm:prob_random} can be considered as a discrete confirmation of a result of Blekherman~\cite{Blek}, which implies that when the degree is even and at least four, there are ``significantly more'' nonnegative polynomials than sos polynomials as the number of variables tends to infinity.

\subsubsection{Computational experiments on random graphs} \label{subsubsec:comput_exper}

For a fixed constant $p \in (0,1)$, although Theorem~\ref{thm:prob_random} might require $n$ to be large in order to obtain random graphs $G_{n,p}$ with $p_{G_{n,p}}(x)$ not sos, we observe computationally that this phenomenon occurs for relatively small values of $n$. To demonstrate this, for each value of $n$ and $p$ given in Table \ref{tab:nonneg_but_not_sos}, we generate 100 random graphs on $n$ vertices and with edge probability $p$. We report a lower bound on the number of times the (nonnegative) polynomial $p_{G_{n,p}}(x)$ is not sos, which by Corollary~\ref{cor:pgx_nonnegative_sos} holds if and only if $\omega(G_{n,p}) < \vartheta'(\overline{G_{n,p}})$. This lower bound is obtained by counting the number of times $\vartheta'(\overline{G_{n,p}})$ is not an integer, which implies $\omega(G_{n,p}) < \vartheta'(\overline{G_{n,p}})$ as $\omega(G_{n,p})$ is an integer. We observe that when $n \geq 150$, the (nonnegative) polynomial $p_{G_{n,p}}(x)$ is almost never sos.
\begin{table}[h]
\centering
\begin{tabular}{|| c || c | c | c | c | c | c | c ||} 
 \hline
 & $n = 25$ & $n = 50$ & $n = 75$ & $n = 100$ & $n = 125$ & $n = 150$ & $n = 175$ \\ [0.25ex] 
 \hline
 \hline
$p=0.1$ & 14 & 68 & 43 & 12 & 17 & 93 & 97  \\  [0.5ex] 
 \hline
$p=0.3$ & 46 & 65 & 66 & 81 & 100 & 100 & 100  \\ [0.5ex] 
 \hline
$p=0.5$ & 44 & 78 & 96 & 100 & 100 & 100 & 100  \\ [0.5ex] 
 \hline
$p=0.7$ & 45 & 84 & 97 & 99 & 100 & 100 & 100  \\ [0.5ex] 
 \hline
$p=0.9$ & 11 & 71 & 98 & 99 & 100 & 100 & 100  \\ [0.5ex] 
 \hline
\end{tabular}
\vspace{-0.1cm}
\caption{A lower bound on the number of times that 100 randomly generated graphs $G_{n,p}$ satisfy $\omega(G_{n,p}) < \vartheta'(\overline{G_{n,p}})$ (or equivalently make the (nonnegative) polynomial $p_{G_{n,p}}(x)$ not sos)}
\label{tab:nonneg_but_not_sos}
\end{table}

We remark that when the value $\vartheta'(\overline{G})$ is not an integer, for any $k \in \R$ with 
$$\lfloor \vartheta'(\overline{G}) \rfloor \leq k < \vartheta'(\overline{G}),$$
the polynomial $p_{G,k}(x)$ is nonnegative but not sos. Thus, this method suggests a very simple and efficient way of generating random nonnegative polynomials of degree 4 that are not sos.

\subsection{Separating hyperplanes} \label{subsec:separating_hyper}

Let $\Sigma_{n,d}$ denote the set of sos polynomials of degree $d$ in $n$ variables. Since $\Sigma_{n,d}$ is a closed convex set, one can always show that a polynomial $p(x)$ of degree $d$ in $n$ variables does not belong to $\Sigma_{n,d}$ by presenting a hyperplane that separates $p(x)$ from $\Sigma_{n,d}$. For a graph $G$ with $p_G(x) \notin \Sigma_{n,4}$, the following theorem makes this hyperplane explicit and gives a geometric interpretation to an optimal solution of the semidefinite program in \eqref{eqn:sch_number_complement}.


\begin{theorem}
Let $G=(V, E)$ be a graph on $n$ vertices with $p_G(x) \notin \Sigma_{n,4}$. An optimal solution $X \in S_n$ to \eqref{eqn:sch_number_complement} provides a hyperplane that separates $p_G(x)$ from $\Sigma_{n,4}$.
\end{theorem}

\begin{proof}
Let us recall the polynomial $p_G(x)$:
$$p_G(x) = -2 \: \omega(G) \mathlarger{\sum}\limits_{ij \in E(G)} x_i^2x_j^2 + (\omega(G) - 1) \left( \mathlarger{\sum}\limits_{i=1}^{n} x_i^2 \right)^2.$$
Notice that the polynomial $p_G(x)$ consists of monomials $x_i^2x_j^2$ for $i,j=1,\dots,n$. Let $V \mathrel{\mathop:}= V(x)$ denote the $n \times n$ symmetric matrix that consists of these monomials where $V_{ij} = x_i^2x_j^2$ for $i,j=1,\dots,n$. Let the matrix of coefficients of $p_G(x)$ in the monomial ordering $V$ be denoted by $M^p \in S_n$, i.e., $p_G(x) = \langle V, M^p \rangle$, where $\langle \cdot, \cdot \rangle$ denotes the standard matrix inner product\footnote{Recall that for two matrices $A, B \in S_n$, we have $\langle A, B \rangle = \text{Tr}(AB) = \sum\limits_{i,j=1}^n A_{ij} B_{ij}$.}. We claim that an optimal solution $X \in S_n$ to~\eqref{eqn:sch_number_complement} satisfies $\langle X, M^p \rangle < 0$ and $\langle X, M^q \rangle \geq 0$ for any sos polynomial $q(x)$ consisting of the monomials in $V$. Here, $M^q \in S_n$ denotes the coefficients of $q(x)$ listed according to the ordering in $V$. Observe that
\[ \hspace{4cm}
M_{ij}^p = \begin{cases} 
      \omega(G)-1 & \text{if} \hspace{0.2cm} i = j, \\
      -1 & \text{if} \hspace{0.2cm} ij \in E(G), \\
      \omega(G)-1 & \text{if} \hspace{0.2cm} ij \notin E(G).
   \end{cases}
\]
Let $X$ be an optimal solution to \eqref{eqn:sch_number_complement}. Then, we have $$X \geq 0, \quad X \succeq 0, \quad \text{Tr}(X)=1, \quad X_{ij} = 0 \text{ for } ij \notin E(G), \quad 2 \sum\limits_{ij \in E(G)} X_{ij} = \vartheta'(\overline{G}) -1,$$
where the last equality follows since
$$\vartheta'(\overline{G}) = \text{Tr}(JX) = \sum\limits_{i,j=1}^n X_{ij} = 2 \sum\limits_{ij \in E(G)} X_{ij} + 2 \sum\limits_{ij \notin E(G)} X_{ij} + \sum\limits_{i=1}^n X_{ii} = 2 \sum\limits_{ij \in E(G)} X_{ij} + 1.$$
Now, observe that
$$\langle X, M^p \rangle = (\omega(G)-1) \text{Tr}(X) - 2 \sum\limits_{ij \in E(G)} X_{ij} = (\omega(G) - 1) - (\vartheta'(\overline{G}) -1) = \omega(G) - \vartheta'(\overline{G}).$$
Since $p_G(x)$ is not sos, by Corollary~\ref{cor:pgx_nonnegative_sos}\eqref{part_bb}, we have $\omega(G) < \vartheta'(\overline{G})$, and therefore $\langle X, M^p \rangle < 0$. 

Consider now an arbitrary sos polynomial $q(x)$ consisting of the monomials in $V$. By~\cite[Section~5]{ParriloThesis} (see also \cite[Lemma 3.5]{ChoiLam1}), we must have $M^q = P + N$, where $P \succeq 0$ and $N \geq 0$. Since $X \succeq 0$ and $X \geq 0$, it follows that $\langle X, M^q \rangle = \langle X, P \rangle + \langle X, N \rangle \geq 0$.
\end{proof}

\subsection{Convexity of the polynomial $p_{G,k}(x)$} \label{subsec:convexity}

In addition to nonnegativity, convexity\footnote{Recall that a polynomial $p(x)$ in $n$ variables is convex if and only if its Hessian $H(x)$, i.e., the $n \times n$ symmetric matrix of its second derivatives, is psd for all $x$.} is another fundamental property of polynomials. The study of the relationship between the set of convex polynomials and the set of sos polynomials is a subject of active research. In~\cite{Blekher}, Blekherman showed that there are convex forms\footnote{A convex form is nonnegative since it vanishes together with its gradient at the origin.} that are not sos, although the problem of constructing an explicit example remained open until recently. In~\cite{Saunderson}, Saunderson provided the first such example. Saunderson's form is of degree 4 and has 272 variables and it is known that such an example does not exist among forms of degree 4 in less than 5 variables~\cite{Bachir}. Given that we have shown in Sections \ref{subsec:graphs}, \ref{subsec:graph_operations}, \ref{subsec:random_graphs} how various graphs $G$ can lead to degree-4 nonnegative forms $p_G(x)$ that are not sos, it is natural to ask whether it is possible to construct \emph{convex} forms that are not sos through the polynomial $p_G(x)$, or more generally $p_{G,k}(x)$. In this section, we show that these polynomials are unfortunately always ``closer'' to being sos than to being convex.

Recall the definitions of the quartic forms $p_G(x)$ and $p_{G,k}(x)$ from \eqref{eqn:our_polynomial} and \eqref{eqn:pgkx}. It is easy to see that if $G$ is a graph with no edge, then $p_{G,k}(x) = (k - 1) \Bigg( \mathlarger{\sum}\limits_{i=1}^{|V(G)|} x_i^2 \Bigg)^2$, which is convex if and only if $k \geq 1$. It turns out that this is the only case where $p_{G,k}(x)$ is convex.

\begin{lemma}
Let $G$ be a graph with at least one edge. Then, the polynomial $p_{G,k}(x)$ is not convex for any $k$.
\label{lem:not_convex}
\end{lemma}

It follows that the polynomial $p_{G,k}(x)$ is convex if and only if $G$ has no edge and $k \geq 1$ (in which case $p_{G,k}(x)$ is clearly sos). In particular, the polynomial $p_G(x)$ is convex if and only if $G$ has no edge (in which case $p_G(x) = 0$). 

Instead of proving Lemma \ref{lem:not_convex}, we prove a more general statement in Lemma \ref{lem:more_general_convex}. Even though the polynomial $p_{G,k}(x)$ does not directly lead to a convex form that is not sos, one might still hope to obtain such a form by considering the following family of polynomials:
$$p_{G,k,\gamma}(x) \mathrel{\mathop:}= p_{G,k}(x) + \gamma \left( \mathlarger{\sum}\limits_{i=1}^{|V(G)|} x_i^2 \right)^2.$$
Indeed, starting with a pair $(G, k)$ for which the nonnegative polynomial $p_{G,k}(x)$ is not sos, by increasing the value of the scalar $\gamma$, one might hope that the polynomial $p_{G,k,\gamma}(x)$ becomes convex before it becomes sos.\footnote{This approach is related to the comparison between two conic programming-based lower bounds on the minimum value of $p_{G,k}(x)$ on the unit sphere that can be obtained by inner approximating the cone of nonnegative forms with the cones of convex forms and sos forms respectively.} The next lemma shows that this can never happen.

\begin{lemma}
Let $G$ be a graph with at least one edge. If the polynomial $p_{G,k,\gamma}(x)$ is convex for some $k$, then $\gamma \geq 1$ (in which case the polynomial $p_{G,k,\gamma}(x)$ is sos).
\label{lem:more_general_convex}
\end{lemma}

\begin{proof}
Let $G = (V,E)$ be a graph with at least one edge and with vertex set $V = \{1, \dots, n\}$. We may assume without of loss of generality that vertex $1$ is adjacent to vertex $2$. Let $H(x)$ denote the Hessian of the polynomial $p_{G,k, \gamma}(x)$. We have
$$H_{1,1}(x) = 12(k-1+\gamma)x_1^2 + 4k \mathlarger{\sum}\limits_{1i \notin E} x_i^2 + 4(\gamma -1) \mathlarger{\sum}\limits_{i=2}^n x_i^2.$$
Observe that $H_{1,1}(0,1,0,\dots,0) = 4(\gamma -1)$. Hence, regardless the value of $k$, the Hessian $H(x)$ cannot not globally psd when $\gamma < 1$. This proves that if $p_{G,k,\gamma}(x)$ is convex, then $\gamma \geq 1$.

Observe also that $p_{G,k,1}(x) = k \mathlarger{\sum}\limits_{i=1}^{n} x_i^4 + 2k \mathlarger{\sum}\limits_{ij \notin E} x_i^2x_j^2$. Therefore, $p_{G,k,\gamma}(x)$ is sos for $\gamma \geq 1$.
\end{proof}

\newpage
\section{Subsets of Sum of Squares Polynomials and Their Graph-Theoretic Interpretations} \label{sec:smaller_subsets_of_Cn}

As observed in the proof of Lemma~\ref{lem:nonneg_if_clique}, the following formulations of $\omega(G)$ and $\vartheta'(\overline{G})$ can be obtained through nonnegativity and sum of squares conditions on the polynomial $p_{G,k}(x)$ defined in \eqref{eqn:pgkx}:
\vspace{-0.7cm}
\begin{multicols}{2}
\begin{equation}
\begin{aligned}
\omega(G) \:\: = \:\:\:\:
& \min\limits_{k \in \R}
& & k \\
& \text{s.t.}
&& p_{G,k}(x) \text{ is nonnegative,}
\end{aligned}
\label{eq:omega_nonnegative}
\end{equation}
\hspace{0.5cm}
\begin{equation}
\begin{aligned}
\vartheta'(\overline{G}) \:\: = \:\:\:\:
& \min\limits_{k \in \R}
& & k \\
& \text{s.t.}
&& p_{G,k}(x) \text{ is sos.}
\end{aligned}
\label{eq:theta_sos2}
\end{equation}
\end{multicols}

\noindent It is therefore natural to wonder what graph parameters some specific subsets of sos polynomials would lead to. In this section, we consider certain well-studied subsets of sos polynomials and examine the bounds that optimization over these subsets produces on the clique number of a graph. We then characterize the graphs for which these bounds are tight for all induced subgraphs.

\subsection{Dsos and sdsos polynomials}\label{subsec:dsos_sdsos}

Recall that a polynomial $p(x)$ in $n$ variables and of degree $2d$ is sos if and only if there exists a matrix $Q \in S_n^+$ such that $p(x) = z(x)^TQz(x)$, where $z(x)$ is the vector of monomials of degree up to $d$~\cite{ChoiLamRez,ParriloThesis}. In~\cite{AhmMaj}, the condition that the matrix $Q$ be psd is replaced with stronger conditions for the purpose of obtaining subsets of sos polynomials that one can optimize over more efficiently. Two such conditions arise from the notions of diagonally dominant and scaled diagonally dominant matrices.  A matrix $A \in S_n$ is
\begin{itemize}
\itemsep0em
\item \emph{diagonally dominant} (dd) if $A_{ii} \geq \sum_{j \neq i} |A_{ij}|$ for every $i =1,\dots,n$,
\item \emph{scaled diagonally dominant} (sdd) if there exists a diagonal matrix $D$, with positive diagonal entries, such that $DAD$ is diagonally dominant. 
\end{itemize}
We refer to the set of $n \times n$ dd (resp. sdd) matrices as $DD_n$ (resp. $SDD_n$). By Gershgorin's circle theorem \cite{Gersh}, we have $DD_n \subseteq SDD_n \subseteq S_n^+$. We say that a polynomial $p(x)$ is \emph{diagonally dominant sum of squares} (dsos) (resp. \emph{scaled diagonally dominant sum of squares} (sdsos)) if there exists a dd (resp. sdd) matrix $Q$ such that $p(x) = z(x)^TQz(x)$.
It turns out that one can optimize a linear function over the set of dsos (resp. sdsos) polynomials intersected with an affine subspace using linear programming (resp. second-order cone programming)~\cite{AhmMaj}. Motivated in part by this fact, many researchers have studied these and related sets in recent years (see, e.g., \cite{SongPar, Gouveia1, Gouveia2, RVZ, Murray, KGNZ} and references therein). These sets also have natural interpretations in the polynomial language; for example ``sdsos'' polynomials are exactly sums of binomial squares, which were studied in the algebra community in early papers of Reznick~\cite{Reznick}, Choi, Lam, and Reznick \cite{CLR}, and Robinson~\cite{Robinson}.


Since dsos/sdsos polynomials form more tractable subsets of sos  polynomials, in view of~\eqref{eq:omega_nonnegative} and~\eqref{eq:theta_sos2}, it is natural to wonder the graph parameters that they produce. For a graph $G$, let us define the following parameters:
\vspace{-0.8cm}
\begin{multicols}{2}
\begin{equation}
\begin{aligned}
\tau(G) \: \mathrel{\mathop:}= \:\:
& \min\limits_{k \in \R}
& & k \\
& \text{s.t.}
&& p_{G,k}(x) \text{ is dsos,}
\end{aligned}
\label{eq:tau_dsos}
\end{equation}
\hspace{0.7cm}
\begin{equation}
\begin{aligned}
\gamma(G) \: \mathrel{\mathop:}= \:\:
& \min\limits_{k \in \R}
& & k \\
& \text{s.t.}
&& p_{G,k}(x) \text{ is sdsos.}
\end{aligned}
\label{eq:gamma_sdsos}
\end{equation}
\end{multicols}

\noindent In this subsection, we show that for any graph $G$, we have
\begin{itemize}
\itemsep0em
\item $\tau(G) = \Delta(G) + 1$, where $\Delta(G)$ is the maximum degree of $G$,
\item $\gamma(G) = \lambda_{\text{max}}(A) + 1$, where $\lambda_{\text{max}}(A)$ is the largest eigenvalue of the adjacency matrix $A$ of $G$.\footnote{In~\cite[Section 4.2]{AhmMaj}, two hierarchies of upper bounds on the clique number based on dsos (resp. sdsos) polynomials are proposed. Our results in this section concern the first level of these hierarchies.} 
\end{itemize}
Since dsos polynomials form a subset of sdsos polynomials, we have $\gamma(G) \leq \tau(G)$ for any graph $G$, and therefore we obtain the well-known inequality $\lambda_{\text{max}}(A) \leq \Delta(G)$ from spectral graph theory (see, e.g., Section 3 in \cite{Spectra}). Together with \eqref{eq:omega_nonnegative} and \eqref{eq:theta_sos2}, we have
$$\omega(G) \leq \vartheta'(\overline{G}) \leq \lambda_{\text{max}}(A) + 1 \leq \Delta(G) + 1.$$

\begin{remark}
Recall the formulations of $\omega(G)$ and $\vartheta'(\overline{G})$ given in \eqref{eqn:omega_copos} and \eqref{eqn:theta_sos}, respectively. As stated in Section \ref{sec:sos_perfect}, the set $\mathcal{C}_n$ (resp. $\mathcal{K}_n$) can equivalently be described as the set of matrices $M \in S_n$ for which the polynomial $p_M(x)$ given in \eqref{eqn:p_M_x} is nonnegative (resp. sos). Parrilo~\cite[Section 5]{ParriloThesis} observed (see also \cite[Lemma 3.5]{ChoiLam1}) that for a matrix $M \in S_n$, the polynomial $p_M(x)$ in \eqref{eqn:p_M_x} is sos if and only if $M \in S_n^+ + N_n$, i.e., that $\mathcal{K}_n = S_n^+ + N_n$. Following the proof of Parrilo and using~\cite[Theorem 3.4 and Theorem 3.6]{AhmMaj}, it is not difficult to verify that the polynomial $p_M(x)$ is sdsos (resp. dsos) if and only if $M \in SDD_n + N_n$ (resp. $M \in DD_n + N_n$). Hence, $\gamma(G)$ and $\tau(G)$ can equivalently be formulated as:
\vspace{-0.8cm}
\begin{multicols}{2}
\begin{equation}
\begin{aligned}
\tau(G) = \:
& \min\limits_{k \in \R}
& & k \\
& \text{s.t.}
&& k(I + \overline{A}) - J \in DD_n + N_n,
\end{aligned}
\label{eqn:tau_dd_n}
\end{equation}
\hspace{0.5cm}
\begin{equation}
\hspace{-0.8cm}
\begin{aligned}
\gamma(G) = \:
& \min\limits_{k \in \R}
& & k \\
& \text{s.t.}
&& k(I + \overline{A}) - J \in SDD_n + N_n.
\end{aligned}
\label{eqn:gamma_sdd_n}
\end{equation}
\end{multicols}
\end{remark}

\begin{theorem}
For any graph $G$, we have $\tau(G) = \Delta(G) + 1$.
\label{thm:dd_max_deg}
\end{theorem}

\begin{proof}
Let $G$ be a graph on $n$ vertices and let $\overline{A}$ denote the adjacency matrix of $\overline{G}$. Let $k \in \R$ be feasible to \eqref{eqn:tau_dd_n}. Then, we have $k(I + \overline{A}) - J = D + N$, where $D \in DD_n$ and $N \in N_n$. Note that the matrix $k(I + \overline{A}) - J$ has entries $k-1$ or $-1$. More precisely, 
\[ \hspace{3cm}
\big( k(I + \overline{A}) - J \big)_{ij} = \begin{cases} 
      k-1 & \text{if} \hspace{0.2cm} i = j, \\
      -1 & \text{if} \hspace{0.2cm} ij \in E(G), \\
      k-1 & \text{if} \hspace{0.2cm} ij \notin E(G).
   \end{cases}
\]
Since $N$ is a nonnegative matrix, for every $ij \in E(G)$, we must have $D_{ij} \leq -1$. Then, since $D$ is a dd matrix, for every $i \in V(G)$, we must have $k-1 \geq \text{deg}(i)$, where $\text{deg}(i)$ denotes the degree of the vertex $i$ in $G$. Thus, for any feasible $k \in \R$, we have $k \geq \Delta(G) +1$, and therefore $\tau(G) \geq \Delta(G) +1$. 

It remains to show that $k = \Delta(G) +1$ is feasible to \eqref{eqn:tau_dd_n}.\footnote{This was also observed in \cite[Theorem 5.1]{AhmDasGeor}.} Indeed, if $k = \Delta(G) +1$, let
\[ \hspace{1cm}
D_{ij} = \begin{cases} 
      \Delta(G) & \text{if} \hspace{0.2cm} i = j, \\
      -1 & \text{if} \hspace{0.2cm} ij \in E(G), \\
      0 & \text{if} \hspace{0.2cm} ij \notin E(G),
   \end{cases}
\hspace{2.5cm}
N_{ij} = \begin{cases} 
      0 & \text{if} \hspace{0.2cm} i = j, \\
      0 & \text{if} \hspace{0.2cm} ij \in E(G), \\
      \Delta(G) & \text{if} \hspace{0.2cm} ij \notin E(G).
   \end{cases}
\]
Then, $k(I + \overline{A}) - J = D + N$, where $D \in DD_n$ and $N \in N_n$.
\end{proof}

We now move to the parameter $\gamma(G)$. Recall the following characterization of the largest eigenvalue of a symmetric $n \times n$ matrix $A$:
\begin{equation*}
\hspace{5cm}
\begin{aligned}
\lambda_{\text{max}}(A) \:\: = \:\:\:
& \min\limits_{k \in \R}
& & k \\
& \text{s.t.}
&& kI - A \in S_n^+.
\end{aligned}
\label{eqn:largest_eig_sdp}
\end{equation*}
The following lemma gives a refined characterization of the largest eigenvalue of the adjacency matrix of a graph.

\begin{lemma}
\label{lem:largest_eig_sdd}
Let $G$ be a graph on $n$ vertices with adjacency matrix $A$. Then,
\begin{equation*}
\hspace{5cm}
\begin{aligned}
\lambda_{\text{max}}(A) \:\: = \:\:\:
& \min\limits_{k \in \R}
& & k \\
& \text{s.t.}
&& kI - A \in SDD_n.
\end{aligned}
\label{eqn:largest_eig_adj_sdd}
\end{equation*}
\end{lemma}

\begin{proof}
Since $SDD_n \subseteq S_n^+$, it is enough to show that for $k = \lambda_{\text{max}}(A)$, we have $kI - A \in SDD_n$. Assume first that $G$ is a connected graph. Let $v = (v_1, \dots, v_n)^T$ be an eigenvector that corresponds to $\lambda_{\text{max}}(A)$. Then, by the Perron-Frobenius\footnote{Recall that the adjacency matrix of a connected graph is irreducible and that the Perron-Frobenius theorem applies to nonnegative irreducible matrices, that is, the largest eigenvalue of an irreducible nonnegative matrix is positive and the corresponding eigenvector can be chosen to be positive.} theorem \cite{Perron, Frobenius}, we have $\lambda_{\text{max}}(A) > 0$ and we may assume that $v_i > 0$ for $i=1,\dots,n$. Let $D = \text{Diag}(v_1, \dots, v_n)$, where $\text{Diag}(v_1, \dots, v_n)$ denotes the $n \times n$ diagonal matrix with the vector $(v_1, \dots, v_n)^T$ on its diagonal. We claim that $D (\lambda_{\text{max}}(A) I - A) D$ is dd, and thus $\lambda_{\text{max}}(A)I - A \in SDD_n$. We have 
$$D (\lambda_{\text{max}}(A) I - A) D = \lambda_{\text{max}}(A) \cdot \text{Diag}(v_1^2, \dots, v_n^2) - A |_{v_iv_j: \: ij \in E(G)},$$
where $A |_{v_iv_j: \: ij \in E(G)}$ denotes the adjacency matrix $A$ with 1's replaced by $v_iv_j$ for $ij \in E(G)$. Hence, the matrix $D (\lambda_{\text{max}}(A) I - A) D$ is dd if and only if $$\lambda_{\text{max}}(A) \cdot v_i^2 \geq \sum\limits_{ij \in E(G)} v_iv_j$$ for every $i=1,\dots,n$, which holds if and only if $\lambda_{\text{max}}(A) \cdot v_i \geq \sum_{ij \in E(G)} v_j$ for every $i=1,\dots,n$. But since $Av = \lambda_{\text{max}}(A) v$, we have $\lambda_{\text{max}}(A) \cdot v_i = \sum_{ij \in E(G)} v_j$ for every $i=1,\dots,n$. This completes the proof for the case when $G$ is a connected graph.

Next, assume that $G$ has several connected components $G_1, \dots, G_r$, with adjacency matrices $A_1, \dots, A_r$ and with $|V(G_i)| = n_i$ for $i=1, \dots, r$. Then, $A$ is a block diagonal matrix with diagonal blocks $A_1, \dots, A_r$, and 
\begin{equation*}
\hspace{-1cm}
\begin{aligned}
\lambda_{\text{max}}(A) \: = \:
\max\limits_{1\leq i \leq r} \lambda_{\text{max}}(A_i) \: =  \:
\max\limits_{1\leq i \leq r}
\min\limits_{k \in \R} \quad
& k \\
\text{s.t.} \quad
& kI - A_i \in SDD_{n_i}
\end{aligned}
\begin{aligned}
\: = \:
\min\limits_{k \in \R} \quad
& k \\
\text{s.t.} \quad
& kI - A_i \in SDD_{n_i}, \: i=1,\dots,r
\end{aligned}
\end{equation*}
\begin{equation*}
\hspace{9.18cm}
\begin{aligned}
\: = \: 
\min\limits_{k \in \R} \quad
& k \\
\text{s.t.} \quad
& kI - A \in SDD_n.
\end{aligned}
\end{equation*}
This completes the proof.
\end{proof}

\begin{theorem}
For any graph $G$, we have $\gamma(G) = \lambda_{\text{max}}(A) + 1$.
\label{thm:sdd_lambda_max}
\end{theorem}

\begin{proof}
Let $G$ be a graph on $n$ vertices, and let $A$ and $\overline{A}$ be the adjacency matrices of $G$ and $\overline{G}$ respectively. Since $k(I + \overline{A}) - J = (k-1) I - A + (k-1) \overline{A}$, by \eqref{eqn:gamma_sdd_n}, we have
\begin{equation}
\hspace{3.5cm}
\begin{aligned}
\gamma(G) \:\: = \:\:\:\:
& \min\limits_{k \in \R}
& & k \\
& \text{s.t.}
&& (k-1) I - A + (k-1) \overline{A} \in SDD_n + N_n.
\end{aligned}
\label{eqn:gamma_sdd_n_2}
\end{equation}
By Lemma \ref{lem:largest_eig_sdd}, we also have
\begin{equation}
\hspace{2.8cm}
\begin{aligned}
\lambda_{\text{max}}(A) + 1 \:\: = \:\:\:
& \min\limits_{k \in \R}
& & k \\
& \text{s.t.}
&& (k-1) I - A \in SDD_n.
\end{aligned}
\label{eqn:largest_eig_adj_sdd_2}
\end{equation}
We claim that 
\begin{equation*}
\hspace{0.5cm}
\begin{aligned}
& \min\limits_{k \in \R}
& & k \\
& \text{s.t.}
&& (k-1) I - A \in SDD_n
\end{aligned}
\hspace{0.5cm}
\begin{aligned}
= \hspace{0.7cm}
& \min\limits_{k \in \R}
& & k \\
& \text{s.t.}
&& (k-1) I - A + (k-1) \overline{A} \in SDD_n + N_n.
\end{aligned}
\end{equation*}
Since $(k-1) \overline{A} \in N_n$ for any $k \geq 1$, any feasible solution to \eqref{eqn:largest_eig_adj_sdd_2} is also feasible to \eqref{eqn:gamma_sdd_n_2}. Next, we show that any feasible solution to \eqref{eqn:gamma_sdd_n_2} is also feasible to \eqref{eqn:largest_eig_adj_sdd_2}. Assume for the sake of contradiction that for some $k \in \R$, we have $(k-1) I - A + (k-1) \overline{A} \in SDD_n + N_n$, but $(k-1) I - A \notin SDD_n$. Then, there exists a diagonal matrix $D$ with positive diagonal entries, a dd matrix $M$, and a nonnegative matrix $N$, such that $D ((k-1) I - A) D + (k-1) D\overline{A}D = M +N$. We have
\[ \hspace{1cm}
\Big (D ((k-1) I - A) D + (k-1) D\overline{A}D \Big)_{ij} = \begin{cases} 
      (k-1)d_i^2 & \text{if} \hspace{0.2cm} i = j, \\
      -d_id_j & \text{if} \hspace{0.2cm} ij \in E(G), \\
      (k-1)d_id_j & \text{if} \hspace{0.2cm} ij \notin E(G).
   \end{cases}
\]
Since $N$ is a nonnegative matrix, for every $ij \in E(G)$, we must have $M_{ij} \leq -d_id_j$. Then, since $M$ is a dd matrix, for every $i \in V(G)$, we must have $(k-1)d_i^2 \geq \sum_{ij \in E(G)} d_id_j$. But then the matrix $D ((k-1) I - A) D$ is a dd matrix since we have
\[ \hspace{2cm}
(D ((k-1) I - A) D)_{ij} = \begin{cases} 
      (k-1)d_i^2 & \text{if} \hspace{0.2cm} i = j, \\
      -d_id_j & \text{if} \hspace{0.2cm} ij \in E(G), \\
      0 & \text{if} \hspace{0.2cm} ij \notin E(G).
   \end{cases}
\]
Hence, $(k-1) I - A \in SDD_n$, a contradiction.
\end{proof}

Next, we characterize the graphs $G$ for which $\omega(H) = \tau(H)$ or $\omega(H) = \gamma(H)$ for every induced subgraph $H$ of $G$. A path on three vertices, denoted by $P_3$, is a graph with vertex set $\{v_1, v_2, v_3\}$ and edge set $\{v_1v_2, v_2v_3\}$; see Figure \ref{fig:P3_and_complement} for the graph $P_3$ and its complement $\overline{P_3}$. It is easy to observe that a graph is $P_3$-free if and only if it is a disjoint union of cliques (i.e., a graph whose connected components are complete graphs). Taking the complement, it follows that a graph is $\overline{P_3}$-free if and only if it is complete multipartite (i.e., a graph whose vertices can be partitioned into independent sets in such a way that all edges across independent sets are present).
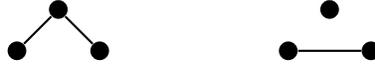
\begin{figure}[h]
\begin{center}
\begin{tikzpicture}[scale=0.22]

\node[inner sep=2.5pt, fill=black, circle] at (0,0)(v1){}; 
\node[inner sep=2.5pt, fill=black, circle] at (2.5,2.5)(v2){}; 
\node[inner sep=2.5pt, fill=black, circle] at (5, 0)(v3){}; 

\draw[black, thick] (v1) -- (v2);
\draw[black, thick] (v2) -- (v3);

\end{tikzpicture}
\hspace{2cm}
\begin{tikzpicture}[scale=0.22]

\node[inner sep=2.5pt, fill=black, circle] at (0,0)(v1){}; 
\node[inner sep=2.5pt, fill=black, circle] at (2.5,2.5)(v2){}; 
\node[inner sep=2.5pt, fill=black, circle] at (5, 0)(v3){}; 

\draw[black, thick] (v1) -- (v3);

\end{tikzpicture}

\end{center}
\vspace{-0.5cm}
\caption{The graph $P_3$ (left) and its complement $\overline{P_3}$ (right)}
\label{fig:P3_and_complement}
\end{figure}

\begin{corollary}
For any graph $G$, the following are equivalent:
\begin{enumerate}[(a)]
\itemsep0em
\item\label{item:tau} $\omega(H) = \tau(H)$ for every induced subgraph $H$ of $G$,
\item\label{item:gamma} $\omega(H) = \gamma(H)$ for every induced subgraph $H$ of $G$,
\item\label{item:graph} $G$ is a disjoint union of cliques.
\end{enumerate}
\end{corollary}

\begin{proof}
Since $\gamma(H) \leq \tau(H)$ for any graph $H$, \eqref{item:tau} clearly implies \eqref{item:gamma}. We first show that \eqref{item:gamma} implies \eqref{item:graph}. Suppose $G$ is not a disjoint union of cliques. Then, $G$ contains $P_3$, and we have 
$$2 = \omega(P_3) < \gamma(P_3) = \lambda_{\text{max}}(A_{P_3}) + 1 = \sqrt{2} + 1,$$
where $A_{P_3}$ is the adjacency matrix of $P_3$ and the equality $\gamma(P_3) = \lambda_{\text{max}}(A_{P_3}) + 1$ follows from Theorem~\ref{thm:sdd_lambda_max}. This shows that $G$ has an induced subgraph $H$ with $\omega(H) \neq \gamma(H)$.

Next, we show that \eqref{item:graph} implies \eqref{item:tau}. Let $G$ be a disjoint union of cliques. Clearly, $\omega(G) = \Delta(G) +1$ and by Theorem \ref{thm:dd_max_deg}, we have $\omega(G) = \tau(G)$. Since every induced subgraph of $G$ is also a disjoint union of cliques, we have $\omega(H) = \tau(H)$ for every induced subgraph $H$ of $G$.
\end{proof}

\subsection{Another structured subset of sos polynomials}\label{subsec:Sn+_psd}

For a graph $G$ and for an arbitrary scalar $k \in \R$, let us define the following quadratic polynomial:
$$\hat p_{G,k}(x) \mathrel{\mathop:}= -2 k \mathlarger{\sum}\limits_{ij \in E(G)} x_ix_j + (k - 1) \left( \mathlarger{\sum}\limits_{i=1}^{|V(G)|} x_i \right)^2.$$
Equivalently, the polynomial $\hat p_{G,k}(x)$ is obtained from $p_{G,k}(x)$ (defined in \eqref{eqn:pgkx}) by replacing $x_i^2$ with $x_i$ for $i=1,\dots,n$. We then define the following parameter:
\begin{equation}
\hspace{4.7cm}
\begin{aligned}
\rho(G) \: \mathrel{\mathop:}= \:\:
& \min\limits_{k \in \R}
& & k \\
& \text{s.t.}
&& \hat p_{G,k}(x) \text{ is nonnegative.}
\end{aligned}
\label{eq:rho_psd}
\end{equation}

Since $\hat p_{G,k}(x)$ is a quadratic polynomial, it is nonnegative if and only if it is sos. Thus, whenever $\hat p_{G,k}(x)$ is nonnegative, the polynomial $p_{G,k}(x_1, \dots, x_n) = \hat p_{G,k}(x_1^2, \dots, x_n^2)$ is sos. Hence, in view of \eqref{eq:omega_nonnegative} and \eqref{eq:theta_sos2}, we have $\omega(G) \leq \vartheta'(\overline{G}) \leq \rho(G)$.

\begin{remark}
Recall the formulations of $\omega(G)$ and $\vartheta'(\overline{G})$ in \eqref{eqn:omega_copos} and \eqref{eqn:theta_sos}, and the fact that $\mathcal{K}_n = S_n^+ + N_n$. We remark that the polynomial $\hat p_{G,k}(x)$ is nonnegative if and only if $k(I + \overline{A}) - J \in S_n^+$. Hence, $\rho(G)$ can equivalently be formulated as:
\begin{equation}
\hspace{4.8cm}
\begin{aligned}
\rho(G) = \:
& \min\limits_{k \in \R}
& & k \\
& \text{s.t.}
&& k(I + \overline{A}) - J \in S_n^+.
\end{aligned}
\label{eqn:rho_psd}
\end{equation}
(This upper bound on $\omega(G)$ has also been mentioned but not carefully analyzed in~\cite[Section 2.1]{BFL}.) One can also define an upper bound on $\omega(G)$ by replacing the set $S_n^+$ in \eqref{eqn:rho_psd} with $N_n$. However, it is easy to observe that $k(I + \overline{A}) - J \in N_n$ for some $k \in \R$ if and only if $\overline{A} = J - I$, which holds if and only if the graph $G$ has no edge.
\end{remark}


The result of this subsection is the following theorem.

\begin{theorem}
The parameter $\rho(G)$ is finite if and only if $G$ is complete multipartite, in which case $\rho(G) = \omega(G)$.
\label{thm:Sn+_comp_multi}
\end{theorem}

\begin{proof}
Suppose first that $G$ is not complete multipartite. We show that the matrix $k(I + \overline{A}) - J$ in~\eqref{eqn:rho_psd} is not psd for any $k \in \R$, and thus $\rho(G)$ is not finite. Since $G$ is not complete multipartite, it contains $\overline{P_3}$. Consider the submatrix of $k(I + \overline{A}) - J$ that corresponds to $\overline{P_3}$: 
$$
  \begin{pmatrix}
     k-1 & k-1 & -1 \\
    k-1 & k-1 & k-1 \\
    -1 & k-1 & k-1
  \end{pmatrix}.
$$
Suppose this matrix is psd for some $k \in \R$. Then, by nonnegativity of its principal minors, we have $k-1 \geq 0$, $(k-1)^2 - 1 \geq 0$, and $-k^2(k-1) \geq 0$. The first and the third inequalities together imply that $k=1$, which contradicts the second inequality.

Conversely, suppose $G$ is complete multipartite. We show that the matrix $k (I + \overline{A}) - J$ is psd for $k=\omega(G)$. Since $\omega(G) \leq \rho(G)$, this would prove that $\rho(G) = \omega(G)$. Let $k=\omega(G)$. Since $G$ is complete multipartite, $V(G)$ can be partitioned into $\omega(G)$ independent sets $V(G) = I_1 \cup \dots \cup I_k$. For $i=1,\dots, k$, let $|I_i| = m_i$. Since
\[ \hspace{3cm}
\big( k(I + \overline{A}) - J \big)_{ij} = \begin{cases} 
      k-1 & \text{if} \hspace{0.2cm} i = j, \\
      -1 & \text{if} \hspace{0.2cm} ij \in E(G), \\
      k-1 & \text{if} \hspace{0.2cm} ij \notin E(G),
   \end{cases}
\]
we have
$$
k(I + \overline{A}) - J = 
  \begin{pmatrix}
      (k-1) J_{m_1} & -J_{m_1m_2} & \dots & -J_{m_1m_k} \\
      -J_{m_2m_1}  & (k-1)J_{m_2} & \dots & -J_{m_2m_k} \\
      \vdots & \vdots  & \ddots & \vdots \\
      -J_{m_km_1} & -J_{m_km_2} & \dots & (k-1)J_{m_k} \\
  \end{pmatrix},
$$
where $J_{m_i}$ and $J_{m_im_j}$ respectively denote the $m_i \times m_i$ and $m_i \times m_j$ all-ones block matrices. Then, observe that the matrix $k(I + \overline{A}) - J$ can be expressed as
$$
  \begin{pmatrix}
      L_{m_1} & 0 & \dots & 0 \\
      0  & L_{m_2} & \dots & 0 \\
      \vdots & \vdots  & \ddots & \vdots \\
      0 & 0 & \dots & L_{m_k} \\
  \end{pmatrix}
  \begin{pmatrix}
      (k-1) U_{m_1} & -U_{m_1m_2} & \dots & -U_{m_1m_k} \\
      -U_{m_2m_1}  & (k-1)U_{m_2} & \dots & -U_{m_2m_k} \\
      \vdots & \vdots  & \ddots & \vdots \\
      -U_{m_km_1} & -U_{m_km_2} & \dots & (k-1)U_{m_k} \\
  \end{pmatrix}
    \begin{pmatrix}
      L_{m_1}^T & 0 & \dots & 0 \\
      0  & L_{m_2}^T & \dots & 0 \\
      \vdots & \vdots  & \ddots & \vdots \\
      0 & 0 & \dots & L_{m_k}^T \\
  \end{pmatrix},
$$
where $L_{m_i}$ is an $m_i \times m_i$ matrix with 1's on its diagonal and on its first column, and 0's everywhere else; and $U_{m_i}$ and $U_{m_im_j}$ are respectively the $m_i \times m_i$ and $m_i \times m_j$ matrices with 1 on the first element of the first row, and 0's everywhere else. Since the matrix in the middle is psd,\footnote{After removing rows and columns consisting of 0's, the matrix in the middle reduces to the $k \times k$ matrix $kI - J$, which has eigenvalues 0 with multiplicity 1, and $k$ with multiplicity $k-1$.} it follows that the matrix $k (I + \overline{A}) - J$ is psd.
\end{proof}

It follows from Theorem~\ref{thm:Sn+_comp_multi} that a graph $G$ is complete multipartite if and only if $\omega(G) = \rho(G)$, which holds if and only if $\omega(H) = \rho(H)$ for every induced subgraph $H$ of $G$.

\section{Future Research Directions}\label{sec:revisit_perfect}


The algebraic characterization of perfect graphs presented in this paper enables an interesting interplay between structural graph theory and nonnegative polynomials (or copositive matrices). In Section~\ref{sec:nonnegative_but_not_sos}, we exploited certain results from graph theory to systematically construct nonnegative polynomials that are not sos, a task which is of interest to the algebraic geometry and polynomial optimization communities. We believe that future research can also transfer ideas in the other direction. As an example, by using facts from linear algebra and optimization, the following graph-theoretic statement follows easily from Corollary~\ref{cor:G_perfect} (without using the strong perfect graph theorem).


\begin{proposition}
\label{prop:NLOC}
A graph with no odd cycles of length 5 or more is perfect.
\end{proposition}

\begin{proof}
Consider a graph $G = (V,E)$ with no odd cycles of length 5 or more. By \cite[Theorem 1]{DrewJohnson}, the optimal value of \eqref{eqn:sch_number_complement} does not change if the two constraints $X\geq0$ and $X\succeq 0$ are replaced with the stronger constraint that $X = BB^T$ for some nonnegative matrix $B$. By \cite[Theorem 2.2]{DKLP}, the optimal value of \eqref{eqn:sch_number_complement} after this replacement is $\omega(G)$. Since these claims hold for every induced subgraph of $G$, it follows from Corollary~\ref{cor:G_perfect} that $G$ is perfect. 
\end{proof}

We hope that future research can similarly provide algebraic proofs of other structural results concerning perfect graphs. Some examples are mentioned in Section~\ref{subsec:revisit_structural}. Other future research directions are discussed in Section~\ref{subsec:other_research}.



\subsection{Algebraic reformulations of structural results concerning perfect graphs} \label{subsec:revisit_structural}

In view of Theorem \ref{thm:main_thm_2}, we reformulate a number of results from the theory of perfect graphs as statements about sum of squares proofs of nonnegativity of certain polynomials.

\medskip
\noindent \textit{\textbf{The weak perfect graph theorem}} \cite{Lovasz2}:

A graph is perfect if and only if its complement is perfect.
\medskip

\noindent The proof of the weak perfect graph theorem was given by Lov\'asz in 1972 and relies on the following lemma which is interesting in its own right. Recall that \emph{replicating} a vertex $v$ of a graph means adding a new vertex to the graph and making it adjacent to $v$ and to all the neighbors of $v$. 

\medskip
\noindent \textit{\textbf{The replication lemma}} \cite{Lovasz2}:

If $G$ is a perfect graph and $G'$ is obtained from $G$ by replicating a vertex of $G$, then $G'$ is perfect.
\medskip

\noindent Other proofs of the weak perfect graph theorem can be obtained by self-complementary characterizations of perfect graphs. The following is one such characterization due to Lov\'asz.

\medskip
\noindent \textit{\textbf{Lov\'asz's characterization of perfect graphs}} \cite{Lovasz3}:

A graph $G$ is perfect if and only if every induced subgraph $H$ of $G$ satisfies $\alpha(H)\omega(H) \geq |V(H)|$.
\medskip

\noindent Lastly, the following (self-complementary) characterization of perfect graphs was proven in~\cite{SPGT} more than forty years after it was conjectured in~\cite{Berge61}.

\medskip
\noindent \textit{\textbf{The strong perfect graph theorem}} \cite{SPGT}: 

A graph is perfect if and only if it does not contain an odd hole or an odd antihole.
\medskip

We now invoke Theorem \ref{thm:main_thm_2} to provide algebraic reformulations of the four structural results above. From~\eqref{eqn:our_polynomial}, recall the definition of the polynomial $p_{G}(x)$ associated with a graph $G$.

\medskip
\noindent \textit{\textbf{The weak perfect graph theorem (algebraic reformulation)}}:

If for every induced subgraph $H$ of a graph $G$, $p_H(x)$ is sos, then $p_{\overline{G}}(x)$ is sos.
\medskip

\begin{remark}
The statement ``if $p_G(x)$ is sos, then $p_{\overline{G}}(x)$ is sos'' is not true in general. To see this, let $G$ be the the graph in Figure~\ref{fig:an_example}. Then, by Corollary~\ref{cor:pgx_nonnegative_sos}\eqref{part_bb}, the polynomial $p_G(x)$ is sos since $\omega(G) = \vartheta'(\overline{G}) = 3$, while the polynomial $p_{\overline{G}}(x)$ is not sos since $\omega(\overline{G}) = 2 < \sqrt{5} = \vartheta'(G)$.
\label{rem:remark6}
\end{remark}

\begin{figure}[h]
\begin{center}
\begin{tikzpicture}[scale=0.22]

\node[inner sep=2.5pt, fill=black, circle] at (0,4.5)(v1){}; 
\node[inner sep=2.5pt, fill=black, circle] at (3.804226, 1.2361)(v2){}; 
\node[inner sep=2.5pt, fill=black, circle] at (2.351141, -3.2361)(v3){}; 
\node[inner sep=2.5pt, fill=black, circle] at (-2.351141, -3.2361)(v4){}; 
\node[inner sep=2.5pt, fill=black, circle] at (-3.804226, 1.2361)(v5){};
\node[inner sep=2.5pt, fill=black, circle] at (-1.5, 1)(v6){};

\draw[black, thick] (v1) -- (v2);
\draw[black, thick] (v1) -- (v5);
\draw[black, thick] (v2) -- (v3);
\draw[black, thick] (v3) -- (v4);
\draw[black, thick] (v4) -- (v5);
\draw[black, thick] (v1) -- (v6);
\draw[black, thick] (v4) -- (v6);
\draw[black, thick] (v5) -- (v6);

\end{tikzpicture}
\end{center}
\vspace{-0.5cm}
\caption{The graph associated with Remark~\ref{rem:remark6}}
\label{fig:an_example}
\end{figure}
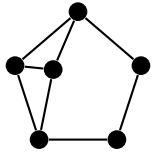

\noindent \textit{\textbf{The replication lemma (algebraic reformulation)}}:

Let $G'$ be the graph obtained from a graph $G$ by replicating a vertex of $G$. If for every induced subgraph $H$ of $G$, $p_H(x)$ is sos, then $p_{G'}(x)$ is sos.
\medskip

\begin{remark}
The statement ``if $p_G(x)$ is sos, then $p_{G'}(x)$ is sos'' is not true in general. To see this, let $G$ be the graph in Figure~\ref{fig:lovasz_replicate} (left), and $G'$ be the graph obtained from $G$ by replicating the vertex $v$ as shown in Figure~\ref{fig:lovasz_replicate} (right). One can verify that $\omega(G) = \vartheta'(\overline{G}) = 3$ while $\omega(G') = 3 < 3.196 < \vartheta'(\overline{G'})$. By Corollary~\ref{cor:pgx_nonnegative_sos}\eqref{part_bb}, the polynomial $p_G(x)$ is sos while the polynomial $p_{G'}(x)$ is not.
\label{rem:remark7}
\end{remark}

\begin{figure}[h]
\begin{center}
\begin{tikzpicture}[scale=0.22]

\node[inner sep=2.5pt, fill=black, circle] at (0,4.5)(v1){}; 
\node[label=right:{\scriptsize{$v$}}, inner sep=2.5pt, fill=black, circle] at (3.804226, 1.2361)(v2){}; 
\node[inner sep=2.5pt, fill=black, circle] at (2.351141, -3.2361)(v3){}; 
\node[inner sep=2.5pt, fill=black, circle] at (-2.351141, -3.2361)(v4){}; 
\node[inner sep=2.5pt, fill=black, circle] at (-3.804226, 1.2361)(v5){};
\node[inner sep=2.5pt, fill=black, circle] at (0, -7)(v6){};
\node[inner sep=2.5pt, fill=black, circle] at (-1.5, 1)(v7){};

\draw[black, thick] (v1) -- (v2);
\draw[black, thick] (v1) -- (v5);
\draw[black, thick] (v2) -- (v3);
\draw[black, thick] (v3) -- (v4);
\draw[black, thick] (v4) -- (v5);
\draw[black, thick] (v3) -- (v6);
\draw[black, thick] (v4) -- (v6);
\draw[black, thick] (v1) -- (v7);
\draw[black, thick] (v4) -- (v7);
\draw[black, thick] (v5) -- (v7);

\end{tikzpicture}
\hspace{2cm}
\begin{tikzpicture}[scale=0.22]

\node[inner sep=2.5pt, fill=black, circle] at (0,4.5)(v1){}; 
\node[label=right:{\scriptsize{$v$}}, inner sep=2.5pt, fill=black, circle] at (3.804226, 1.2361)(v2){}; 
\node[inner sep=2.5pt, fill=black, circle] at (2.351141, -3.2361)(v3){}; 
\node[inner sep=2.5pt, fill=black, circle] at (-2.351141, -3.2361)(v4){}; 
\node[inner sep=2.5pt, fill=black, circle] at (-3.804226, 1.2361)(v5){};
\node[inner sep=2.5pt, fill=black, circle] at (0, -7)(v6){};
\node[inner sep=2.5pt, fill=black, circle] at (-1.5, 1)(v7){};
\node[inner sep=2.5pt, fill=black, circle] at (1.5, 1)(v8){};

\draw[black, thick] (v1) -- (v2);
\draw[black, thick] (v1) -- (v5);
\draw[black, thick] (v2) -- (v3);
\draw[black, thick] (v3) -- (v4);
\draw[black, thick] (v4) -- (v5);
\draw[black, thick] (v3) -- (v6);
\draw[black, thick] (v4) -- (v6);
\draw[black, thick] (v1) -- (v7);
\draw[black, thick] (v4) -- (v7);
\draw[black, thick] (v5) -- (v7);
\draw[black, thick] (v1) -- (v8);
\draw[black, thick] (v2) -- (v8);
\draw[black, thick] (v3) -- (v8);

\end{tikzpicture}
\end{center}
\vspace{-0.6cm}
\caption{The graphs $G$ (left) and $G'$ (right) associated with Remark~\ref{rem:remark7}}
\label{fig:lovasz_replicate}
\end{figure}
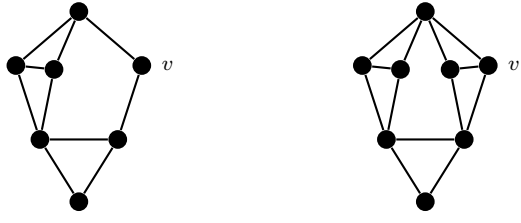

The following is a reformulation of the nontrivial direction of Lov\'asz's characterization of perfect graphs. (The other direction is immediate since $\alpha(G)\chi(G) \geq |V(G)|$ for any graph $G$.)

\medskip
\noindent \textit{\textbf{Lov\'asz's characterization of perfect graphs (algebraic reformulation)}}:

If every induced subgraph $H$ of a graph $G$ satisfies $\alpha(H)\omega(H) \geq |V(H)|$, then $p_G(x)$ is sos.

\medskip
\begin{remark}
The statement ``if $G$ satisfies $\alpha(G)\omega(G) \geq |V(G)|$, then $p_G(x)$ is sos'' is not true in general. As an example, consider the graph $G$ obtained from $C_5$ by adding a vertex adjacent to every vertex of $C_5$. We have $\omega(G) = 3$, $\alpha(G) = 2$, and $|V(G)| = 6$. However, by Corollary~\ref{cor:pgx_nonnegative_sos}\eqref{part_bb}, the polynomial $p_G(x)$ is not sos since $\vartheta'(\overline{G}) = \sqrt{5} +1$.
\end{remark}

The following is a reformulation of the nontrivial direction of the strong perfect graph theorem. (As mentioned in Section~\ref{subsec:perfect_graphs}, the other direction is immediate.)

\medskip
\noindent \textit{\textbf{The strong perfect graph theorem (algebraic reformulation - polynomial version)}}:
 
If a graph $G$ does not contain an odd hole or an odd antihole, then $p_G(x)$ is sos.
\medskip

\noindent Considering \eqref{eqn:omega_copos} and \eqref{eqn:theta_sos}, together with Corollary~\ref{cor:pgx_nonnegative_sos}\eqref{part_bb} and the fact that $\mathcal{K}_n = S_n^+ + N_n$, another reformulation of the strong perfect graph theorem is the following.

\medskip
\noindent \textit{\textbf{The strong perfect graph theorem (algebraic reformulation - matrix version)}}:

If a graph $G$ does not contain an odd hole or an odd antihole, then $\omega(G) (I + \overline{A}) - J \in S_n^+ + N_n$.

\medskip
\noindent We hope that the above algebraic reformulations lead to new (and ideally simpler) proofs of these statements, possibly using tools from convex algebraic geometry and semidefinite optimization. This would be particularly exciting in the case of the strong perfect graph theorem, whose original proof in~\cite{SPGT} is long and highly nontrivial.

\subsection{Other research directions}
\label{subsec:other_research}

We end with a few other research directions that could be of interest.



\subsubsection{Sos decompositions for subclasses of perfect graphs}

By Theorem \ref{thm:main_thm_2}, for any perfect graph $G$, the polynomial $p_G(x)$ is sos. Can one write down an explicit sos decomposition for some subclasses of perfect graphs? For example, for a complete graph $G$ on $n$ vertices, we have the following two possible sos decompositions:
$$p_G(x) = \sum_{ij \in E(G)} (x_i^2 - x_j^2)^2,$$
$$p_G(x) = \sum\limits_{i =1}^{n-1} \frac{n}{(n-i+1)(n-i)} \Big ((n-i)x_i^2 - x_{i+1}^2 - x_{i+2}^2 - \dots - x_n^2 \Big )^2.$$

Similarly, for a bipartite graph $G$ whose vertex set is partitioned into two (independent) sets $A=\{1,2,\dots,p\}$ and $B=\{p+1,p+2,\dots,n\}$, we have the following sos decomposition:
$$p_G(x) = \left( \sum_{i=1}^{p} x_i^2 - \sum_{i=p+1}^{n} x_i^2 \right)^2 + \sum_{i \in A, j \in B, ij \notin E(G)} (2 x_ix_j)^2.$$

Can one similarly write sos decompositions for other subclasses of perfect graphs, e.g., for chordal graphs? More interestingly, does an sos decomposition of $p_G(x)$ contain useful information about the perfect graph $G$?

\subsubsection{Algebraic imperfection ratio}\label{sec:imperfection_ratio}

It would be interesting to study the relationship between the so-called \emph{imperfection ratio} of a graph (see~\cite{GerMcD} for a definition) and an algebraic notion we define below which measures how close a given graph is to being sos-perfect.

Let $p^{\text{min}}_G$ and $p^{\text{max}}_G$ respectively denote the minimum and the maximum value of the nonnegative quartic form $p_G(x)$ on the unit sphere:
\vspace{-0.9cm}
\begin{multicols}{2}
\begin{equation*}
\begin{aligned}
\hspace{2.7cm}
p^{\text{min}}_G \:\:\: \mathrel{\mathop:}= \:\:\:
& \min\limits_{||x||^2=1}
& & p_G(x),
\end{aligned}
\end{equation*}\hspace{0.5cm}
\begin{equation*}
\begin{aligned}
p^{\text{max}}_G \:\:\: \mathrel{\mathop:}= \:\:\:
& \max\limits_{||x||^2=1}
& & p_G(x).
\end{aligned}
\end{equation*}
\end{multicols}

\noindent In general, a (computationally tractable) lower bound on $p^{\text{min}}_G$ can be obtained by
\vspace{-0.1cm}
\begin{equation*}
\hspace{4.9cm}
\begin{aligned}
p^{\text{sos}}_G \:\:\: \mathrel{\mathop:}= \:\:\:
& \max\limits_{\gamma \in \R}
& & \gamma \\
& \text{s.t.}
&& p_G(x) - \gamma ||x||^{4} \text{ is sos}.
\end{aligned}
\label{eq:min_sos_dist}
\end{equation*}
One way to measure the quality of $p^{\text{sos}}_G$ as a lower bound on $p^{\text{min}}_G$ is through the ``normalized'' quantity
$(p^{\text{max}}_G - p^{\text{sos}}_G)/(p^{\text{max}}_G - p^{\text{min}}_G)$. Taking induced subgraphs into consideration, we propose the following as to measure how close a given graph is to being sos-perfect:
\vspace{-0.2cm}
\begin{equation}
\hspace{4.8cm}
\begin{aligned}
\text{AIMP}(G) \mathrel{\mathop:}= \max_{H \subseteq G} \left\{ \frac{p^{\text{max}}_H - p^{\text{sos}}_H}{p^{\text{max}}_H - p^{\text{min}}_H} \right \}.
\end{aligned}
\label{eqn:imp_sos_perf}
\end{equation}
Here, the maximum is taken over all induced subgraphs $H$ of $G$. The following is a simple observation, which we give without proof.
\begin{observation}
For any graph $G$, we have $p^{\text{min}}_G = 0$ and $p^{\text{max}}_G = \omega(G)-1$.
\label{obs:min_max_pG}
\end{observation}

Notice that since $p^{\text{min}}_G = 0$, we have $p^{\text{sos}}_G \leq 0$ for every graph $G$. By Observation \ref{obs:min_max_pG}, we have
$$\text{AIMP}(G) = \max_{H \subseteq G} \left\{ 1 - \frac{p^{\text{sos}}_H}{\omega(H)-1} \right \}.$$

From~\eqref{eqn:imp_sos_perf}, it is easy to see that $\text{AIMP}(H) \leq \text{AIMP}(G)$ for every induced subgraph $H$ of $G$. Moreover, $\text{AIMP}(G) \geq 1$ for every graph $G$, and $\text{AIMP}(G) = 1$ if and only if $G$ is sos-perfect (i.e., perfect). (These properties of $\text{AIMP}(G)$ hold also for the imperfection ratio that is proposed in~\cite{GerMcD}.) Can one classify interesting families of graphs for which $\text{AIMP}(G)$ is bounded above by a given constant? Is it possible to efficiently approximate the clique number or the chromatic number for these graph families, e.g., by semidefinite programming?

\subsubsection{$r$-sos-perfect graphs}\label{sec:denominator}

Theorem~\ref{thm:main_thm} and the construction of the polynomial $p_G(x)$ provide a framework for studying the relationship between subsets of nonnegative polynomials and subsets or supersets of perfect graphs. In this paper, we establish this relationship for sos, sdsos, and dsos polynomials, but the same question for certain natural supersets of sos polynomials remains open, as we describe next. 

From Artin's solution~\cite{Artin} to Hilbert's 17th problem, we know that for every nonnegative form $p(x)$, there exists a nonzero sos form $q(x)$ such that $p(x)q(x)$ is sos.  (The representation of $p(x)$ as the ratio of two sos forms algebraically certifies nonnegativity of $p(x)$.) For an even nonnegative integer $r$, let us call a graph $G$ \emph{$r$-sos-perfect} if for every induced subgraph $H$ of $G$, there exists a nonzero degree-$r$ sos form $q_H(x)$ with $p_H(x) q_H(x)$ sos. (Here, $p_H(x)$ is defined as in~\eqref{eqn:our_polynomial}.) With this definition, $0$-sos-perfect graphs are precisely perfect graphs. What are $r$-sos-perfect graphs for $r \geq 2$? What is the minimum $r$ that makes odd holes or odd antiholes (or some other families of imperfect graphs) $r$-sos-perfect? Does the answer relate to the algebraic imperfection ratio of these graphs, and if so how?


\section*{Acknowledgements}

We thank Abraar Chaudhry for insightful discussions around Proposition \ref{prop:NLOC} and the anonymous referees for their careful reading of the manuscript.



\begin{thebibliography}{100}

\bibitem{AhmDasGeor}
A.A.~Ahmadi, S.~Dash, G.~Hall, ``Optimization over structured subsets of positive semidefinite matrices via column generation'', {\it Discrete Optimization}, 24 (2017), 129--151.

\bibitem{AhmMaj}
A.A.~Ahmadi, A.~Majumdar, ``DSOS and SDSOS optimization: More tractable alternatives to sum of squares and semidefinite optimization'', {\it SIAM Journal on Applied Algebra and Geometry}, 3(2) (2019), 193--230.

\bibitem{Artin}
 E.~Artin, ``\"Uber die Zerlegung Definiter Funktionen in Quadrate'', {\it Abhandlungen aus dem Mathematischen Seminar der Universit\"at Hamburg}, 5 (1927), 100--115.

\bibitem{Berge61}
C.~Berge, ``F\"arbung von Graphen, deren s\"amtliche bzw. deren ungerade Kreise starr sind'', {\it Wiss. Z. Martin-Luther-Univ. Halle-Wittenberg Math.-Natur. Reihe}, 10 (1961), 114--114.

\bibitem{BHT}
R.G.~Bland, H.-C.~Huang, L.E.~Trotter Jr, ``Graphical properties related to minimal imperfection'', {\it Discrete Mathematics}, 27 (1979), 11--22.

\bibitem{Blek}
G.~Blekherman, ``There are significantly more nonnegative polynomials than sums of squares'', {\it Israel Journal of Mathematics}, 153 (2006), 355--380.

\bibitem{Blekher}
G.~Blekherman, ``Convex forms that are not sums of squares'', arXiv preprint arXiv:0910.0656, (2009).


\bibitem{BPT}
G.~Blekherman, P.A.~Parrilo, R.R.~Thomas, ``Semidefinite Optimization and Convex Algebraic Geometry'', {\it SIAM}, (2012).

\bibitem{BelaBol}
B.~Bollob\'as, ``Random Graphs'', Second Ed., Cambridge University Press, Cambridge, (2001).


\bibitem{BDKRQT}
I.M.~Bomze, M.~D\"ur, E.~de Klerk, C.Roos, A.J.~Quist, T.~Terlaky, ``On copositive programming and standard quadratic optimization problems'', {\it Journal of Global Optimization}, 18 (4) (2000), 301--320.

\bibitem{BFL}
I.M.~Bomze, F.~Frommlet, M.~Locatelli, ``Gap, cosum and product properties of the $\theta'$ bound on the clique number'', {\it Optimization}, 59 (7) (2010), 1041--1051.

\bibitem{BroHae}
A.E.~Brouwer, W.H.~Haemers, ``Spectra of graphs'', Springer, New York, (2012).

\bibitem{ChoiLam1}
M.D.~Choi, T.Y.~Lam, ``An old question of Hilbert'', {\it Queen's Paper in Pure and Applied Math}, 46 (1977).

\bibitem{ChoiLam2}
M.D.~Choi, T.Y.~Lam, ``Extremal positive semidefinite forms'', {\it Mathematische Annalen}, 231 (1) (1977), 1--18.

\bibitem{ChoiLamRez80}
M.D.~Choi, T.Y.~Lam, B.~Reznick, ``Real zeros of positive semidefinite forms. I'', {\it Mathematische Zeitschrift}, 171 (1) (1980), 1--26.

\bibitem{CLR}
M.-D.~Choi, T.-Y.~Lam, B.~Reznick, ``Even symmetric sextics'', {\it Mathematische Zeitschrift}, 195 (1987), 559--580.

\bibitem{ChoiLamRez}
M.D.~Choi, T.Y.~Lam, B.~Reznick. ``Sums of squares of real polynomials'',  In {\it Proceedings of Symposia in Pure Mathematics}, 58 (1995), 103--126.


\bibitem{CRST}
M.~Chudnovsky, N.~Robertson, P.~Seymour, R.~Thomas, ``Progress on perfect graphs'', {\it Mathematical Programming}, 97 (1) (2003), 405--422.

\bibitem{SPGT}
M.~Chudnovsky, N.~Robertson, P.~Seymour, R.~Thomas, ``The strong perfect graph theorem'', {\it Annals of Mathematics}, 164 (2006), 51--229.


\bibitem{Chvatal}
V.~Chv\'atal, ``On certain polytopes associated with graphs'', {\it Journal of Combinatorial Theory (B)}, 18 (1975), 138--154.

\bibitem{CGPW}
V.~Chv\'atal, R.L.~Graham, A.F.~Perold, S.~Whitesides, ``Combinatorial designs related to the strong perfect graph conjecture'', {\it Discrete Mathematics}, 26 (1979), 83--92.

\bibitem{Coja-Oghlan}
A.~Coja-Oghlan, ``The Lov\'asz number of random graphs'', {\it Combinatorics, Probability and Computing}, 14 (4) (2005), 439--465.

\bibitem{CKLMS}
I.~Csisz\'ar, J.~K\"orner, L.~Lov\'asz, K.~Marton, G.~Simonyi, ``Entropy splitting for antiblocking corners and perfect graphs'', {\it Combinatorica}, 10 (1) (1990), 27--40.

\bibitem{CMRSSW}
T.~Cubitt, L.~Man\v{c}inska, D.E.~Roberson, S.~Severini, D.~Stahlke, A.~Winter, ``Bounds on entanglement-assisted source-channel coding via the Lov\'asz $\vartheta$ number and its variants'', {\it IEEE Transactions on Information Theory}, 60 (11) (2014), 7330--7344.

\bibitem{Spectra}
D.M.~Cvetkovi\'c, M.~Doob, H.~Sachs, ``Spectra of Graphs: Theory and Application'', Academic Press, New York, (1980).

\bibitem{DdZ}
P.J.~Dickinson, R.~de Zeeuw, ``Generating irreducible copositive matrices using the stable set problem'', {\it Discrete Applied Mathematics}, 296 (2021), 103--117.

\bibitem{DKLP}
E.~De Klerk, D.V.~Pasechnik, ``Approximation of the stability number of a graph via copositive programming'', {\it SIAM Journal on Optimization}, 12 (4) (2002), 875--892.


\bibitem{DGH}
T.~Do\v{s}lic, M.~Ghorbani, M.A.~Hosseinzadeh, ``The relationships between Wiener index, stability number and clique number of composite graphs'', {\it Bulletin of the Malaysian Mathematical Sciences Society}, 36 (1) (2013), 165--172.

\bibitem{DrewJohnson}
J.H.~Drew, C.R.~Johnson, ``The no long odd cycle theorem for completely positive matrices'', In {\it Random Discrete Structures}, Springer, New York, NY (1996), 103--115.


\bibitem{Bachir}
B.~El Khadir, ``On Sum of Squares Representation of Convex Forms and Generalized Cauchy--Schwarz Inequalities'', {\it SIAM Journal on Applied Algebra and Geometry}, 4 (2) (2020), 377--400.

\bibitem{ErdRen}
P.~Erd\H{o}s, A.~R\'enyi, ``On random graphs I'', {\it Publ. Math. Debrecen}, 6 (1959), 290--297.

\bibitem{Frobenius}
G.~Frobenius, ``\"Uber Matrizen aus nicht negativen Elementen'', {\it S.B. Preuss Acad. Wiss.}, 26 (1912), 456--477.

\bibitem{Fulkerson}
D.R.~Fulkerson, ``Blocking and anti-blocking pairs of polyhedra'', {\it Mathematical Programming}, 1 (1) (1971), 168--194.

\bibitem{Gasparian}
G.S.~Gasparian, ``Minimal imperfect graphs: a simple approach'', {\it Combinatorica}, 16 (1996), 209--212.

\bibitem{GerMcD}
S.~Gerke, C.~McDiarmid, ``Graph imperfection I, II'', {\it Journal of Combinatorial Theory (B)}, 83 (2001), 58--78, 79--101.

\bibitem{Gersh}
S.A.~Gershgorin, ``\"Uber die Abgrenzung der Eigenwerte einer Matrix'', {\it Bulletin de l'Acad\'emie des Sciences de l'URSS.} Classe des sciences math\'ematiques et na (1931), no. 6, 749--754.

\bibitem{GRSS}
C.~Godsil, D.E.~Roberson, R.~\v{S}\'amal, S.~Severini, ``Sabidussi versus Hedetniemi for three variations of the chromatic number'', {\it Combinatorica}, 36 (4) (2016), 395--415.

\bibitem{Gouveia2}
J.~Gouveia, A.~Kovačec, M.~Saee, ``On sums of squares of $k$-nomials'', {\it Journal of Pure and Applied Algebra}, 226 (1) (2022).

\bibitem{GoPaTh}
J.~Gouveia, P.A.~Parrilo, R.R.~Thomas, ``Theta bodies for polynomial ideals'', {\it SIAM Journal on Optimization}, 20 (4) (2010), 2097--2118.

\bibitem{Gouveia1}
J.~Gouveia, T.K.~Pong, M.~Saee, ``Inner approximating the completely positive cone via the cone of scaled diagonally dominant matrices'', {\it Journal of Global Optimization}, 76 (2) (2020), 383--405.


\bibitem{GLS}
M.~Gr\"otschel, L.~Lov\'asz, A.~Schrijver, ``The ellipsoid method and its consequences in combinatorial optimization'', {\it Combinatorica}, 1 (1981), 169--197.


\bibitem{GLS2}
M.~Gr\"otschel, L.~Lov\'asz, A.~Schrijver, ``Geometric algorithms and combinatorial optimization'', {\it Springer-Verlag, Berlin}, (1988).


\bibitem{ghsurvey}
G.~Hall. ``Applications of sum of squares polynomials'', {\it Sum of Squares: Theory and Applications}, Proceedings of Symposia in Applied Mathematics, Vol. 77 (2020).

\bibitem{HanPet}
B.~Hanson, G.~Petridis, ``Refined estimates concerning sumsets contained in the roots of unity'', {\it Proc. Lond. Math. Soc.}, 3 (122) (2021), 353--358.

\bibitem{Hilbert}
D.~Hilbert, ``\"Uber die Darstellung Definiter Formen als Summe von Formenquadraten'', {\it Mathematische Annalen}, 32 (3) (1888), 342--350.



\bibitem{KGNZ}
X.~Kuang, B.~Ghaddar, J.~Naoum-Sawaya, L.F.~Zuluaga, ``Alternative LP and SOCP hierarchies for ACOPF problems'', {\it IEEE Transactions on Power Systems}, 32 (4) (2016), 2828--2836.

\bibitem{Lasserre2}
J.B.~Lasserre, ``Global optimization with polynomials and the problem of moments'', {\it SIAM Journal on Optimization}, 11 (2001), 296--817.

\bibitem{LasserreBook}
J.B.~Lasserre, ``Moments, Positive Polynomials and Their Applications'', {\it Imperial College Press Optimization Series}, vol. 1, Imperial College Press, London, (2010).

\bibitem{Laurentsurvey}
M.~Laurent, ``Sums of squares, moment matrices and optimization over polynomials'', In {\it Emerging applications of algebraic geometry}, Springer, New York, NY (2009), 157--270.

\bibitem{LV1}
M.~Laurent, L.F.~Vargas, ``Finite convergence of sum-of-squares hierarchies for the stability number of a graph'', {\it SIAM Journal on Optimization}, 32(2) (2022), 491--518.

\bibitem{LV2}
M.~Laurent, L.F.~Vargas, ``Exactness of Parrilo's conic approximations for copositive matrices and associated low order bounds for the stability number of a graph'', arXiv preprint arXiv:2109.12876, (2021).

\bibitem{LV3}
M.~Laurent, L.F.~Vargas, ``On the exactness of sum-of-squares approximations for the cone of $5\times 5$ copositive matrices'', {\it Linear Algebra and its Applications}, 651 (2022), 26--50.

\bibitem{Lovasz2}
L.~Lov\'asz, ``Normal hypergraphs and the perfect graph conjecture'', {\it Discrete Mathematics}, 2 (1972), 253--267.

\bibitem{Lovasz3}
L.~Lov\'asz, ``A characterization of perfect graphs'', {\it Journal of Combinatorial Theory (B)}, 13 (1972), 95--98.

\bibitem{Lovasz}
L.~Lov\'asz, ``On the Shannon capacity of a graph'', {\it IEEE Transactions on Information Theory}, 25(1) (1979), 1--7.

\bibitem{Lovasz0}
L.~Lov\'asz, ``Perfect Graphs'', In {\it Selected Topics in Graph Theory 2}, L.W. Beineke and R.J. Wilson ed., Academic Press, (1983), 55--87.

\bibitem{MRR}
R.J.~McEliece, E.R.~Rodemich, H.C.~Rumsey, Jr, ``The Lov\'asz bound and some generalizations'', {\it J. Combinatorics, Inform. Syst. Sci.}, 3 (1978), 134--152.

\bibitem{Motzkin}
T.~Motzkin, ``The arithmetic-geometric inequality'', In {\it Proceedings of Symposium on Inequalities}, (1967), 205--224.

\bibitem{MS}
T.S.~Motzkin, E.G.~Straus, ``Maxima for graphs and a new proof of a theorem of Tur\'an'', {\it Canadian J. Math.}, 17 (1965), 533--540.

\bibitem{Murray}
R.~Murray, V.~Chandrasekaran, A.~Wierman, ``Signomial and polynomial optimization via relative entropy and partial dualization'', {\it Mathematical Programming Computation}, 13 (2) (2021), 257--295.

\bibitem{MurtKaba}
K.G.~Murty, S.N.~Kabadi, ``Some NP-complete problems in quadratic and nonlinear programming'', {\it Mathematical Programming}, 39 (1987), 117--129.

\bibitem{Mycielski}
J.~Mycielski, ``Sur le coloriage des graphes'', {\it Colloquium Mathematicae}, 3 (2) (1955), 161--162.

\bibitem{Padberg}
M.W.~Padberg, ``Perfect zero-one matrices'', {\it Mathematical Programming}, 6 (1974), 180--196.


\bibitem{ParriloThesis}
P.A.~Parrilo, ``Structured semidefinite programs and semialgebraic geometry methods in robustness and optimization'', Ph.D. Thesis, California Institute of Technology, (2000).

\bibitem{ParriloMP}
P.A.~Parrilo, ``Semidefinite programming relaxations for semialgebraic problems'', {\it Mathematical Programming}, 96 (2003), 293--320.




\bibitem{Pecher1}
A.~P\^{e}cher, ``Partitionable graphs arising from near-factorizations of finite groups'', {\it Discrete Mathematics}, 269 (1-3) (2003), 191--218.

\bibitem{Pecher2}
A.~P\^{e}cher, ``Cayley partitionable graphs and near-factorizations of finite groups'', {\it Discrete Mathematics}, 276 (1-3) (2004), 295--311.

\bibitem{PVZ}
J.~Pe\~{n}a, J.~Vera, L.F.~Zuluaga, ``Computing the stability number of a graph via linear and semidefinite programming'', {\it SIAM Journal on Optimization}, 18 (1) (2007), 87--105.

\bibitem{Perron}
O.~Perron, ``Zur Theorie der Matrizen'', {\it Mathematische Annalen}, 64 (1907), 248--263.


\bibitem{Reznick}
B.~Reznick, ``Forms derived from the arithmetic-geometric inequality'', {\it Mathematische Annalen}, 283 (1989), 431--464.

\bibitem{Reznick}
B.~Reznick, ``Some concrete aspects of Hilbert's 17th problem'', In {\it Contemporary Mathematics}, 253 (2000), 251--272.

\bibitem{Robinson}
R.M.~Robinson, ``Some definite polynomials which are not sums of squares of real polynomials'', In {\it Selected Questions of Algebra and Logic}, (1973), 264--282.

\bibitem{RVZ}
L.M.~Roebers, J.C.~Vera, L.F.~Zuluaga, ``Sparse non-SOS Putinar-type Positivstellens\"atze'', arXiv preprint arXiv:2110.10079, (2021).

\bibitem{Robinson}
R.M.~Robinson, ``Some definite polynomials which are not sums of squares of real polynomials'', In {\it Notices of the American Mathematical Society}, 16 (1969), 554--555.

\bibitem{Saunderson}
J.~Saunderson, ``A convex form that is not a sum of squares'', arXiv preprint arXiv:2105.08432, (2021).

\bibitem{Schrijver}
A.~Schrijver, ``A comparison of the Delsarte and Lov\'asz bounds'', {\it IEEE Trans. Inform. Theory}, 25 (4) (1979), 425--429.

\bibitem{SongPar}
D.~Song, P.A.~Parrilo, ``On approximations of the psd cone by a polynomial number of smaller-sized psd cones'', {\it Mathematical Programming}, (2022), 1--53.



\end{thebibliography}
\end{document}